\documentclass{aims}
\usepackage{amsmath}
  \usepackage{paralist}
  \usepackage{graphics} 
  \usepackage{epsfig} 
\usepackage{graphicx}  \usepackage{epstopdf}
 \usepackage[colorlinks=true]{hyperref}
\hypersetup{urlcolor=blue, citecolor=red}

\def\currentvolume{X}

    \def\ppages{X--XX}
     \def\DOI{10.3934/xx.xx.xx.xx}

\newtheorem{theorem}{Theorem}[section]
\newtheorem{corollary}[theorem]{Corollary}

\newtheorem{lemma}[theorem]{Lemma}
\newtheorem{proposition}{Proposition}

\theoremstyle{definition}

\newcommand{\eps}{\varepsilon}

\DeclareMathOperator{\e}{e} 
\DeclareMathOperator{\sgn}{sgn} 
\DeclareMathOperator{\tr}{tr} 
\usepackage{enumerate}
\usepackage{tikz}
\usetikzlibrary{arrows,backgrounds,shapes}
\usepackage{pict2e, curve2e}
\usepackage{caption}

\title[Planar S-systems: permanence] 
      {Planar S-systems: permanence}

\author[B. Boros and J. Hofbauer]{}



\thanks{BB is supported by the Austrian Science Fund (FWF), project P28406.}

\begin{document}
\maketitle

%

\centerline{\scshape Bal\'azs Boros and Josef Hofbauer}
\medskip
{\footnotesize
 \centerline{Faculty of Mathematics, University of Vienna}
 \centerline{Oskar-Morgenstern-Platz 1, 1090 Vienna, Austria}
}

\bigskip


\begin{abstract}
We characterize permanence of planar S-systems. Further, we construct a planar S-system with three limit cycles.
\end{abstract}

\section{Introduction}

An S-system is a dynamical system on the positive orthant for which the right hand side is given by differences of power products (monomials) with real exponents. They were introduced by Savageau \cite{Savageau1969} in the context of biochemical systems theory. In a previous paper \cite{boros:hofbauer:mueller:regensburger:2017} we studied planar S-systems, especially the local and global asymptotic stability of the unique positive equilibrium and also the center problem.

In the present paper we characterize (except for some boundary case) the permanence of planar S-systems. This is done by first transforming planar S-systems into a $3$-dimensional replicator dynamics.

The results will be illustrated for some special cases: Selkov's model for glycolytic oscillations \cite{Selkov68, brechmann:rendall:2018} and the Lotka reactions with generalized mass-action kinetics \cite{dancso:farkas:farkas:szabo:1991,boros:hofbauer:mueller:2017}.

Finally, the obtained results allow us to construct a planar S-system with three limit cycles. This improves the previous studies \cite{dancso:farkas:farkas:szabo:1991} (one limit cycle) and \cite{boros:hofbauer:mueller:regensburger:2017,boros:hofbauer:mueller:regensburger:2017b} (two limit cycles).

\section{Planar S-systems}\label{sec:planar_s_systems}

A planar S-system is given by
\begin{align} \label{eq:ode_S}
\begin{split}
\dot x_1 &= \alpha_1 \, x_1^{g_{11}} x_2^{g_{12}} - \beta_1 \, x_1^{h_{11}} x_2^{h_{12}}  , \\
\dot x_2 &= \alpha_2 \, x_1^{g_{21}} x_2^{g_{22}} - \beta_2 \, x_1^{h_{21}} x_2^{h_{22}}   
\end{split}
\end{align}
with $\alpha_1, \alpha_2, \beta_1, \beta_2 \in \mathbb{R}_+$ and $g_{11},g_{12},g_{21},g_{22},h_{11},h_{12},h_{21},h_{22} \in \mathbb{R}$. Since we allow real exponents, we study the dynamics on the positive quadrant $\mathbb{R}_+^2 = \{(x_1,x_2) \in \mathbb{R}^2 ~|~ x_1>0,x_2>0\}$. Our aim in this paper is to characterize the parameters for which the ODE~\eqref{eq:ode_S} is \emph{permanent}, meaning that there exists a compact subset of $\mathbb{R}^2_+$ that is forward invariant and is a global attractor.

A short calculation shows that there is either $0$, $1$, or infinitely many positive equilibria, and, if there is no positive equilibrium then the system cannot be permanent. We thus concentrate on the case, when there exists a (not necessarily unique) positive equilibrium $(x_1^*,x_2^*)$. Introducing
\begin{align*}
\gamma_1 = \alpha_1 (x_1^*)^{g_{11}-1} (x_2^*)^{g_{12}} \text{ and } \gamma_2 = \alpha_2 (x_1^*)^{g_{21}} (x_2^*)^{g_{22}-1},
\end{align*}
we perform the nonlinear transformation
\begin{align*}
u = \frac{1}{\gamma_1}\log\frac{x_1}{x_1^*} \text{ and } v = \frac{1}{\gamma_2}\log\frac{x_2}{x_2^*}.
\end{align*}
This leads to the ODE
\begin{align} \label{eq:ode_exp}
\begin{split}
\dot u &= \e^{a_1 u + b_1 v} - \e^{a_2 u + b_2 v} , \\
\dot v &= \e^{a_3 u + b_3 v} - \e^{a_4 u + b_4 v} 
\end{split}
\end{align}
with state space $\mathbb{R}^2$, where 
\begin{align} \label{eq:a_b_gamma_g_h}
\begin{aligned}
a_1 &= \gamma_1 (g_{11}-1), &b_1 &= \gamma_2 g_{12}   , \\
a_2 &= \gamma_1 (h_{11}-1), &b_2 &= \gamma_2 h_{12}   , \\
a_3 &= \gamma_1  g_{21}   , &b_3 &= \gamma_2(g_{22}-1), \\
a_4 &= \gamma_1  h_{21}   , &b_4 &= \gamma_2(h_{22}-1).
\end{aligned}
\end{align}

We say that the ODE~\eqref{eq:ode_exp} is \emph{permanent} if there exists a compact subset of $\mathbb{R}^2$ that is forward invariant and is a global attractor. Clearly, the permanence of the ODE~\eqref{eq:ode_S} is equivalent to the permanence of the ODE~\eqref{eq:ode_exp} with \eqref{eq:a_b_gamma_g_h}.

The ODE~\eqref{eq:ode_exp} admits the origin as an equilibrium. The Jacobian matrix $J$ at the origin is given by
\begin{align} \label{eq:J}
J =
\begin{pmatrix}
a_1-a_2 & b_1-b_2 \\
a_3-a_4 & b_3-b_4
\end{pmatrix}.
\end{align}
A short calculation shows that if $\det J = 0$ then the set of equilibria is either a line through the origin or the whole $\mathbb{R}^2$. Thus, the system cannot be permanent for $\det J = 0$. To characterize permanence for the ODE~\eqref{eq:ode_S}, it suffices to characterize those $a_1$, $a_2$, $a_3$, $a_4$, $b_1$, $b_2$, $b_3$, $b_4 \in \mathbb{R}$, for which $\det J \neq 0$ and the ODE~\eqref{eq:ode_exp} is permanent. The aim of this paper is to perform this characterization (except for some boundary case). Crucial for this is the relative position of the four points $P_i=(a_i,b_i)$ for $i = 1, 2, 3, 4$ in the plane. Define the numbers $c_1$, $c_2$, $c_3$, $c_4$ by
\begin{align} \label{eq:c_Delta}
\begin{split}
c_1 &= \Delta(243), \\
c_2 &= \Delta(134), \\
c_3 &= \Delta(142), \\
c_4 &= \Delta(123),
\end{split}
\end{align}
where $\Delta(ijk)=\det(P_j-P_i, P_k-P_i)$ is twice the signed area of the triangle $P_iP_jP_k$. The quantity $\Delta(ijk)$ is thus positive (respectively, negative) if the sequence $P_i$, $P_j$, $P_k$, $P_i$ of points are positively (respectively, negatively) oriented. The quantity $\Delta(ijk)$ is zero if the three points $P_i$, $P_j$, $P_k$ lie on a line. Note also that
\begin{align*}
\Delta(ijk) = \Delta(jki) =  \Delta(kij) = -\Delta(jik) = -\Delta(ikj) = -\Delta(kji)
\end{align*}
and $c_1 + c_2 + c_3 + c_4 = 0$.

Now we show how the sign pattern of $c=(c_1,c_2,c_3,c_4)$ is related to the relative position of the four points $P_1$, $P_2$, $P_3$, $P_4$. There are four qualitatively different situations. (The case $c = (0,0,0,0)$ we ignore, because then the four points $P_1$, $P_2$, $P_3$, $P_4$ are co-linear, contradicting $\det J \neq 0$.)
\begin{enumerate}[(i)]
\item When $\sgn c = (+,+,-,-)$, the four points $P_1$, $P_2$, $P_3$, $P_4$ form a quadrangle with diagonals $P_1P_2$ and $P_3P_4$. The geometric interpretation of the equality $\frac{c_1}{2} + \frac{c_2}{2} = \frac{-c_3}{2} + \frac{-c_4}{2}$ is that the area of this quadrangle can be written as the sum of the areas of the triangles $P_2P_4P_3$ and $P_1P_3P_4$, or alternatively as the sum of the areas of the triangles $P_1P_2P_4$ and $P_1P_3P_2$.
\item When $\sgn c = (+,-,-,-)$, the three points $P_2$, $P_3$, $P_4$ form a triangle and the point $P_1$ lies inside. The geometric interpretation of the equality $\frac{c_1}{2} = \frac{-c_2}{2} + \frac{-c_3}{2} + \frac{-c_4}{2}$ is that the area of this triangle can be written as the sum of the areas of the triangles $P_1P_4P_3$, $P_1P_2P_4$, $P_1P_3P_2$.
\item When $\sgn c = (+,0,-,-)$, the three points $P_2$, $P_3$, $P_4$ form a triangle and the point $P_1$ lies in the edge $P_3P_4$. The geometric interpretation of the equality $\frac{c_1}{2} = \frac{-c_3}{2} + \frac{-c_4}{2}$ is that the area of this triangle can be written as the sum of the areas of the triangles $P_1P_2P_4$, $P_1P_3P_2$.
\item When $\sgn c = (+,0,-,0)$, the three points $P_2$, $P_3$, $P_4$ form a triangle and the point $P_1$ coincides with $P_3$. Clearly, $\frac{c_1}{2} = \frac{-c_3}{2}$, because the triangles $P_3P_2P_4$ and $P_1P_2P_4$ coincide, and thus, their area are the same.
\end{enumerate}

\section{Main results} \label{sec:main_results}

In this section we list the main results of this paper.

The following simple lemma states that permanence of the ODE~\eqref{eq:ode_exp} is possible only under $\det J  > 0$.

\begin{lemma}\label{lem:perm_implies_det>0}
If $\det J \leq 0$ then the ODE~\eqref{eq:ode_exp} is not permanent.
\end{lemma}
\begin{proof}
We discussed in Section \ref{sec:planar_s_systems} that the ODE~\eqref{eq:ode_exp} cannot be permanent under $\det J = 0$.

In case $\det J < 0$, the origin is the unique equilibrium, and therefore, would the system be permanent, the index of the origin is $+1$, contradicting $\det J < 0$. Thus, the system cannot be permanent under $\det J < 0$.
\end{proof}

The following three theorems provide an almost complete characterization of permanence of the ODE~\eqref{eq:ode_exp}. The first one deals with the easier case when the diagonal entries in the Jacobian matrix at the origin are of the same sign or one of them is zero. The third one deals with the more complicated case when the diagonal entries have opposite nonzero sign. The case of a heteroclinic cycle at infinity needs a separate treatment, this is dealt with in the second of these three theorems. The proofs of the latter two are given in Sections \ref{sec:S2R} and \ref{sec:proof_oppositesign}.

\begin{theorem}\label{thm:samesign}
Assume that $J_{11}J_{22} \geq 0$ (i.e., $(a_1-a_2)(b_3-b_4)\geq 0$). Then the following three statements are equivalent.
\begin{enumerate}[(i)]
\item\label{it:samesign_perm} The ODE~\eqref{eq:ode_exp} is permanent.
\item\label{it:samesign_gas} The origin is globally asymptotically stable for the ODE~\eqref{eq:ode_exp}.
\item\label{it:samesign_charact} $\det J > 0$ and one of (A), (B1), (B2) below holds.
\begin{enumerate}
\item[(A)] $\sgn J = \begin{pmatrix} - & * \\ * & - \end{pmatrix}$
\item[(B1)] $\sgn J = \begin{pmatrix} 0 & * \\ * & - \end{pmatrix}$ and $\min(a_3,a_4) \leq a_2 = a_1 \leq \max(a_3,a_4)$
\item[(B2)] $\sgn J = \begin{pmatrix} - & * \\ * & 0 \end{pmatrix}$ and $\min(b_1,b_2) \leq b_4 = b_3 \leq \max(b_1,b_2)$
\end{enumerate}
\end{enumerate}
\end{theorem}
\begin{proof}
First, we prove the implication \eqref{it:samesign_perm} $\Rightarrow$ \eqref{it:samesign_charact}. By Lemma \ref{lem:perm_implies_det>0}, permanence implies that $\det J > 0$. After multiplying the vector field by $e^{-a_1 u - b_4 v}$, its divergence is
\begin{align*}
(a_1-a_2)e^{(a_2-a_1)u+(b_2-b_4)v} + (b_3-b_4)e^{(a_3-a_1)u+(b_3-b_4)v},
\end{align*}
which is negative if at least one of $a_1-a_2$ and $b_3-b_4$ is negative, zero if $a_1-a_2 = b_3-b_4 = 0$, and positive if at least one of $a_1-a_2$ and $b_3-b_4$ is positive. The latter two cases cannot lead to a permanent system as the area is invariant or expanding. This leaves the three sign patterns of $J$ given in (A), (B1), (B2). In the latter two cases, the conditions $\min(a_3,a_4) \leq a_2 = a_1 \leq \max(a_3,a_4)$ and $\min(b_1,b_2) \leq b_4 = b_3 \leq \max(b_1,b_2)$ follow from the boundedness of all solutions, see \cite[Lemma 5 (b2), (c2), (d2), (e2)]{boros:hofbauer:mueller:regensburger:2017b}.

The implication \eqref{it:samesign_charact} $\Rightarrow$ \eqref{it:samesign_gas} follows from \cite[Theorem 3]{boros:hofbauer:mueller:regensburger:2017b}.

Finally, the implication \eqref{it:samesign_gas} $\Rightarrow$ \eqref{it:samesign_perm} is obvious.
\end{proof}

\begin{theorem}\label{thm:heteroclinic}
Assume that $\det J > 0$ and further that either (A) or (B) below holds.
\begin{enumerate}[(A)]
\item\label{it:heteroclinicA} $\sgn J = \begin{pmatrix}*&-\\+&*\end{pmatrix}$ and $\left\{\begin{array}{l}
a_4 \leq \min(a_1,a_2) \leq \max(a_1,a_2) \leq a_3 \\
b_1 \leq \min(b_3,b_4) \leq \max(b_3,b_4) \leq b_2
\end{array}\right.$
\item\label{it:heteroclinicB} $\sgn J = \begin{pmatrix}*&+\\-&*\end{pmatrix}$ and $\left\{\begin{array}{l}
a_3 \leq \min(a_1,a_2) \leq \max(a_1,a_2) \leq a_4 \\
b_2 \leq \min(b_3,b_4) \leq \max(b_3,b_4) \leq b_1
\end{array}\right.$
\end{enumerate}
In case \eqref{it:heteroclinicA}, let
\begin{align*}
L_\infty = (a_3 - a_1)(b_2 - b_3)(a_2 - a_4)(b_4 - b_1) - (b_1 - b_3)(a_2 - a_3)(b_4 - b_2)(a_4 - a_1),
\end{align*}
while in case \eqref{it:heteroclinicB}, let
\begin{align*}
L_\infty = (b_1 - b_3)(a_2 - a_3)(b_4 - b_2)(a_4 - a_1) - (a_3 - a_1)(b_2 - b_3)(a_2 - a_4)(b_4 - b_1).
\end{align*} 
Then the following two statements hold.
\begin{enumerate}[(i)]
\item If $L_\infty > 0$ then the ODE~\eqref{eq:ode_exp} is permanent.
\item If $L_\infty < 0$ then the ODE~\eqref{eq:ode_exp} is not permanent.
\end{enumerate}
\end{theorem}

\begin{theorem}\label{thm:oppositesign}
Assume that $J_{11}J_{22}<0$ (i.e., $(a_1-a_2)(b_3-b_4)<0$) and at least one of the two conditions
\begin{align*}
\min(a_3,a_4) \leq \min(a_1,a_2) &\leq \max(a_1,a_2) \leq \max(a_3,a_4) \text{ and}\\
\min(b_1,b_2) \leq \min(b_3,b_4) &\leq \max(b_3,b_4) \leq \max(b_1,b_2)
\end{align*}
is violated. Then the ODE~\eqref{eq:ode_exp} is permanent if and only if $\det J > 0$ and one of (C1), (C2), (C3), (C4) below holds.
\begin{itemize}
\item[(C1)] $\sgn J = \begin{pmatrix} + & - \\ + & - \end{pmatrix}$, $a_4 \leq a_2 < a_1 \leq a_3$, and either of (C1a), (C1b), or (C1c) below holds
\begin{itemize}
\item[(C1a)] $\sgn(c_3,c_4) = (-,-)$
\item[(C1b)] $\sgn c = (+,-,-,0)$ or $\sgn c = (-,+,0,-)$ and 
\begin{align*}
-\frac{a_1-a_2}{b_1-b_2} < \frac{(L+1)^{L+1}}{L^L},
\text{ where } L = \begin{cases} \frac{c_2}{c_3}, & \text{ if } \sgn c = (+,-,-,0), \\ \frac{c_1}{c_4}, & \text{ if } \sgn c = (-,+,0,-)\end{cases}
\end{align*}
\item[(C1c)] $\sgn c = (+,0,-,0)$ or $\sgn c = (0,+,0,-)$ and
\begin{align*}
-\frac{a_1-a_2}{b_1-b_2} \leq 1 \text{ and } \tr J < 0
\end{align*}
\end{itemize}
\item[(C2)] $\sgn J = \begin{pmatrix} + & + \\ - & - \end{pmatrix}$, $a_3 \leq a_2 < a_1 \leq a_4$, and either of (C2a), (C2b), or (C2c) below holds
\begin{itemize}
\item[(C2a)] $\sgn(c_3,c_4) = (-,-)$
\item[(C2b)] $\sgn c = (-,+,-,0)$ or $\sgn c = (+,-,0,-)$ and 
\begin{align*}
\frac{a_1-a_2}{b_1-b_2} < \frac{(L+1)^{L+1}}{L^L},
\text{ where } L = \begin{cases} \frac{c_1}{c_3}, & \text{ if } \sgn c = (-,+,-,0), \\ \frac{c_2}{c_4}, & \text{ if } \sgn c = (+,-,0,-)\end{cases}
\end{align*}
\item[(C2c)] $\sgn c = (0,+,-,0)$ or $\sgn c = (+,0,0,-)$ and
\begin{align*}
\frac{a_1-a_2}{b_1-b_2} \leq 1 \text{ and } \tr J < 0
\end{align*}
\end{itemize}
\item[(C3)] $\sgn J = \begin{pmatrix} - & - \\ + & + \end{pmatrix}$, $b_1 \leq b_4 < b_3 \leq b_2$, and either of (C3a), (C3b), or (C3c) below holds
\begin{itemize}
\item[(C3a)] $\sgn(c_1,c_2) = (+,+)$
\item[(C3b)] $\sgn c = (0,+,-,+)$ or $\sgn c = (+,0,+,-)$ and 
\begin{align*}
\frac{b_3-b_4}{a_3-a_4} < \frac{(L+1)^{L+1}}{L^L},
\text{ where } L = \begin{cases} \frac{c_4}{c_2}, & \text{ if } \sgn c = (0,+,-,+), \\ \frac{c_3}{c_1}, & \text{ if } \sgn c = (+,0,+,-)\end{cases}
\end{align*}
\item[(C3c)] $\sgn c = (0,+,-,0)$ or $\sgn c = (+,0,0,-)$ and
\begin{align*}
\frac{b_3-b_4}{a_3-a_4} \leq 1 \text{ and } \tr J < 0
\end{align*}
\end{itemize}
\item[(C4)] $\sgn J = \begin{pmatrix} - & + \\ - & + \end{pmatrix}$, $b_2 \leq b_4 < b_3 \leq b_1$, and either of (C4a), (C4b), or (C4c) below holds
\begin{itemize}
\item[(C4a)] $\sgn(c_1,c_2) = (+,+)$
\item[(C4b)] $\sgn c = (0,+,+,-)$ or $\sgn c = (+,0,-,+)$ and 
\begin{align*}
-\frac{b_3-b_4}{a_3-a_4} < \frac{(L+1)^{L+1}}{L^L},
\text{ where } L = \begin{cases} \frac{c_3}{c_2}, & \text{ if } \sgn c = (0,+,+,-), \\ \frac{c_4}{c_1}, & \text{ if } \sgn c = (+,0,-,+)\end{cases}
\end{align*}
\item[(C4c)] $\sgn c = (0,+,0,-)$ or $\sgn c = (+,0,-,0)$ and
\begin{align*}
-\frac{b_3-b_4}{a_3-a_4} \leq 1 \text{ and } \tr J < 0
\end{align*}
\end{itemize}
\end{itemize}
\end{theorem}

With Theorems \ref{thm:samesign}, \ref{thm:heteroclinic}, \ref{thm:oppositesign}, permanence is characterized except in the marginal case $L_\infty = 0$ in Theorem \ref{thm:heteroclinic}. The solution of this remaining case requires more sophisticated techniques.

By \emph{robust permanence} of the ODE~\eqref{eq:ode_exp}, we mean that the ODE remains permanent after small perturbation of the eight parameters $a_i$ and $b_i$ (for $i=1,2,3,4$). The following corollary is a characterization of robust permanence, it is an immediate consequence of the above three theorems.

\begin{corollary}\label{cor:robust_permanence}
The ODE~\eqref{eq:ode_exp} is robustly permanent if and only if $\det J > 0$ and one of (A), (B1), (B2), (B3), (B4), (C1), (C2), (C3), (C4) below holds. (The number $L_\infty$ below in (C1), (C2), (C3), (C4) is defined as in the corresponding case in Theorem \ref{thm:heteroclinic}.)
\begin{itemize}
\item[(A)] $\sgn J = \begin{pmatrix} - & * \\ * & - \end{pmatrix}$
\item[(B1)] $\sgn J = \begin{pmatrix} 0 & - \\ + & - \end{pmatrix}$ and $a_4 < a_2 = a_1 < a_3$
\item[(B2)] $\sgn J = \begin{pmatrix} 0 & + \\ - & - \end{pmatrix}$ and $a_3 < a_2 = a_1 < a_4$
\item[(B3)] $\sgn J = \begin{pmatrix} - & + \\ - & 0 \end{pmatrix}$ and $b_2 < b_4 = b_3 < b_1$
\item[(B4)] $\sgn J = \begin{pmatrix} - & - \\ + & 0 \end{pmatrix}$ and $b_1 < b_4 = b_3 < b_2$
\item[(C1)] $\sgn J = \begin{pmatrix} + & - \\ + & - \end{pmatrix}$, $a_4 < a_2 < a_1 < a_3$, $\sgn(c_3,c_4) = (-,-)$, \\ and if $b_1 < b_3 < b_4 < b_2$ then $L_\infty > 0$
\item[(C2)] $\sgn J = \begin{pmatrix} + & + \\ - & - \end{pmatrix}$, $a_3 < a_2 < a_1 < a_4$, $\sgn(c_3,c_4) = (-,-)$, \\ and if $b_2 < b_3 < b_4 < b_1$ then $L_\infty > 0$
\item[(C3)] $\sgn J = \begin{pmatrix} - & - \\ + & + \end{pmatrix}$, $b_1 < b_4 < b_3 < b_2$, $\sgn(c_1,c_2) = (+,+)$, \\ and if $a_4 < a_1 < a_2 < a_3$ then $L_\infty > 0$
\item[(C4)] $\sgn J = \begin{pmatrix} - & + \\ - & + \end{pmatrix}$, $b_2 < b_4 < b_3 < b_1$, $\sgn(c_1,c_2) = (+,+)$, \\ and if $a_3 < a_1 < a_2 < a_4$ then $L_\infty > 0$
\end{itemize}
\end{corollary}

\section{From S-systems to replicator systems} \label{sec:S2R}

To prove Theorems \ref{thm:heteroclinic} and \ref{thm:oppositesign}, we embed the two-dimensional ODE~\eqref{eq:ode_exp} into a four-dimensional Lotka--Volterra system. Let $z_i = \e^{a_i u + b_i v}$ for $i = 1, 2, 3, 4$. Then 
\begin{align} \label{eq:ode_LV}
\dot z_i = z_i \left[a_i \dot u + b_i \dot v\right] = z_i \left[ a_i (z_1 - z_2) + b_i (z_3 - z_4)\right] \text{ for } i = 1, 2, 3, 4,
\end{align}
which is a 4-dimensional Lotka--Volterra system with matrix
\begin{align*}
\widetilde{A} = \begin{pmatrix}
a_1 & -a_1 & b_1 & - b_1 \\
a_2 & -a_2 & b_2 & - b_2 \\
a_3 & -a_3 & b_3 & - b_3 \\
a_4 & -a_4 & b_4 & - b_4 
\end{pmatrix}.
\end{align*}

Since the ODE~\eqref{eq:ode_LV} is homogeneous, we can reduce the dimension by projecting the dynamics to the 3-dimensional simplex 
\begin{align*}
\Delta_4 = \left\{ x \in \mathbb{R}^4_{\geq0} ~|~ x_1 + x_2 + x_3 + x_4 = 1\right\}.
\end{align*}
Let
\begin{align*}
x_i = \frac{z_i}{z_1 + z_2 + z_3 + z_4} \text{ for } i = 1, 2, 3, 4.
\end{align*}
Then, after division by the positive factor $z_1+z_2+z_3+z_4$, the ODE~\eqref{eq:ode_LV} leads to
\begin{align} \label{eq:ode_Atilde}
\dot x_i = x_i \left[ (\widetilde{A}x)_i - x^\mathsf{T} \widetilde{A} x\right] \text{ for } i = 1, 2, 3, 4
\end{align}
on the simplex $\Delta_4$, the replicator equation with matrix $\widetilde{A}$, see \cite{HS88,HS98}. Adding any multiple of $\mathbf{1}=(1,1,1,1)^\mathsf{T}$ to any column of $\widetilde{A}$ leaves the ODE~\eqref{eq:ode_Atilde} unchanged, therefore, we can replace the ODE~\eqref{eq:ode_Atilde} by the replicator equation on the simplex $\Delta_4$ with matrix
\begin{align} \label{eq:A}
A = \begin{pmatrix} 
0 & a_2-a_1 & b_1-b_3 &  b_4- b_1 \\
a_2-a_1 & 0 & b_2-b_3 &  b_4- b_2 \\
a_3-a_1 & a_2-a_3 & 0  &  b_4- b_3 \\
a_4-a_1 & a_2-a_4 & b_4-b_3 & 0 
\end{pmatrix},
\end{align}
i.e.,
\begin{align} \label{eq:ode_A}
\dot x_i = x_i \left[ (Ax)_i - x^\mathsf{T} A x\right] \text{ for } i = 1, 2, 3, 4.
\end{align}

Note that besides the corners $E_1$, $E_2$, $E_3$, $E_4$ of $\Delta_4$, the points $E_{12} = \left(\frac12,\frac12,0,0\right)$ and $E_{34} = \left(0,0,\frac12,\frac12\right)$ are equilibria of the ODE~\eqref{eq:ode_A}. Also, each point on the line segment connecting $E_{12}$ and $E_{34}$ is an equilibrium. Further, a short calculation shows that $\det J \neq 0$ implies that there is no other equilibrium in $\Delta_4^\circ$, the relative interior of $\Delta_4$.

Let $c \in \mathbb{R}^4$ be as in \eqref{eq:c_Delta}, or explicitly,
\begin{align} \label{eq:c}
\begin{split}
c_1 &= a_3 b_2 - a_2 b_3 + a_2 b_4 - a_4 b_2 + a_4 b_3 - a_3 b_4, \\
c_2 &= a_1 b_3 - a_3 b_1 + a_4 b_1 - a_1 b_4 + a_3 b_4 - a_4 b_3, \\
c_3 &= a_2 b_1 - a_1 b_2 + a_1 b_4 - a_4 b_1 + a_4 b_2 - a_2 b_4, \\
c_4 &= a_1 b_2 - a_2 b_1 + a_3 b_1 - a_1 b_3 + a_2 b_3 - a_3 b_2.
\end{split}
\end{align}
Note that
\begin{align}
\label{eq:c1_c2_detJ} c_1 + c_2 &= +\det J, \\
\label{eq:c3_c4_detJ} c_3 + c_4 &= -\det J
\end{align}
and $c$ is perpendicular to the three vectors $a$, $b$, and $\mathbf{1}$ (a short calculation shows that $\det J \neq 0$ implies that $a$, $b$, and $\mathbf{1}$ are linearly independent). Then $c^\mathsf{T} A = 0$ and hence the function $Q:\Delta_4^\circ\to\mathbb{R}_+$ defined by $Q(x) = \prod_{i=1}^4 x_i^{c_i}$ is a constant of motion for the ODE~\eqref{eq:ode_A}. Indeed, by the choice of $c$,
\begin{align*}
(Q(x))^\cdot=Q(x)\sum_{i=1}^4 c_i\left[(Ax)_i - x^\mathsf{T} Ax\right]= Q(x)\left[c^\mathsf{T} Ax - x^\mathsf{T} Ax \sum_{i=1}^4 c_i\right] = 0.
\end{align*}
The planar S-system \eqref{eq:ode_exp} corresponds to the restriction of the ODE~\eqref{eq:ode_A} to the surface $\{Q = 1\}$. (The origin of the $(u,v)$-plane is mapped to $z = \mathbf{1}$ and $x = \frac{1}{4} \mathbf{1}$. Further, since $\sum_{i=1}^4 c_i = 0$, we have $Q\left(\frac{1}{4} \mathbf{1}\right) = 1$. See Figure~\ref{fig:int_eq} for an illustration.) The shape of the surface $\{Q = 1\}$ depends on the sign pattern of $c$. Let
\begin{align*}
\overline{S} &= \Delta_4 \cap \{Q=1\}, \\
S            &= \Delta_4^\circ \cap \{Q=1\}, \\
\partial S   &= \overline{S} \setminus S.
\end{align*}

\begin{figure}
\begin{center}
\begin{tikzpicture}[radius=0.05]
\coordinate (E1)  at (0   ,0   );
\coordinate (E2)  at (3   ,0   );
\coordinate (E3)  at (1.5 ,2.4 );
\coordinate (E4)  at (1.5*1.25 ,0.72*1.25);
\coordinate (E12) at (1.5 ,0   );
\coordinate (E34) at (1.5*1.125 ,1.65);
\coordinate (E1234) at (1.5*1.0625,0.825);
\draw [semithick] (E1) -- (E2);
\draw [semithick] (E3) -- (E4);
\draw [semithick] (E1) -- (E4);
\draw [semithick] (E4) -- (E2);
\draw [semithick] (E2) -- (E3);
\draw [semithick] (E3) -- (E1);
\draw [blue,semithick] (E12) -- (E34);
\node [below left]  at (E1) {\small $E_1$};
\node [below right] at (E2) {\small $E_2$};
\node [above]       at (E3) {\small $E_3$};
\node [right] at (E4) {\small $E_4$};
\node [below]       at (E12) {\small $E_{12}$};
\node [left]        at (E34) {\small $E_{34}$};
\node [left]        at (E1234) {\small $\frac{1}{4} \mathbf{1}$};
\draw [blue,fill] (E12) circle;
\draw [blue,fill] (E34) circle;
\draw [blue,fill] (E1234) circle;
\end{tikzpicture}
\end{center}
\caption{Each point on the line segment connecting $E_{12}$ and $E_{34}$ is an equilibrium. The planar S-system \eqref{eq:ode_exp} and the origin of the $(u,v)$-plane correspond to the surface $\prod_{i=1}^4 x_i^{c_i}=1$ and $x = \frac{1}{4} \mathbf{1}$ (the midpoint between $E_{12}$ and $E_{34}$), respectively.}
\label{fig:int_eq}
\end{figure}
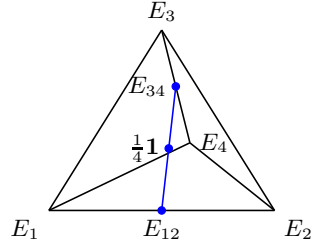

Note that the planar S-system \eqref{eq:ode_exp} is permanent if and only if the restriction of the ODE~\eqref{eq:ode_A} to the surface $\overline{S}$ is \emph{permanent}, meaning there exists a compact subset of $S$ that is forward invariant and attracts every point in $S$ (or, equivalently, $\partial S$ is a repeller in $\overline S$).

We illustrate the usefulness of this embedding by proving Theorem \ref{thm:heteroclinic}. We state and prove here case (A). Case (B) follows from case (A) by time reversal of the ODE~\eqref{eq:ode_exp}, which swaps $P_1$ with $P_2$ and $P_3$ with $P_4$.

\begin{lemma}
Assume that $\det J > 0$ and further that
\begin{align*}
\sgn J = \begin{pmatrix}*&-\\+&*\end{pmatrix} \text{ and }\left\{\begin{array}{l}
a_4 \leq \min(a_1,a_2) \leq \max(a_1,a_2) \leq a_3 \\
b_1 \leq \min(b_3,b_4) \leq \max(b_3,b_4) \leq b_2
\end{array}.\right.
\end{align*}
Let
\begin{align*}
L_\infty = (a_3 - a_1)(b_2 - b_3)(a_2 - a_4)(b_4 - b_1) - (b_1 - b_3)(a_2 - a_3)(b_4 - b_2)(a_4 - a_1).
\end{align*}
Then the following two statements hold.
\begin{enumerate}[(i)]
\item If $L_\infty > 0$ then the ODE~\eqref{eq:ode_exp} is permanent.
\item If $L_\infty < 0$ then the ODE~\eqref{eq:ode_exp} is not permanent (orbits with large initial conditions spiral outwards towards infinity).
\end{enumerate}
\end{lemma}
\begin{proof}
First note that $L_\infty \neq 0$ implies that all the four points $P_1$, $P_2$, $P_3$, $P_4$ are distinct. Further, the assumptions $a_4 \leq \min(a_1,a_2) \leq \max(a_1,a_2) \leq a_3$ and $b_1 \leq \min(b_3,b_4) \leq \max(b_3,b_4) \leq b_2$ yield $\sgn c = (+,+,-,-)$, see \eqref{eq:c_Delta}. Then the surface $S$ is given by $\frac{x_1^{|c_1|}x_2^{|c_2|}}{x_3^{|c_3|}x_4^{|c_4|}}=1$ and thus, $\partial S$ consists of the four edges $\mathcal{F}_{13}$, $\mathcal{F}_{32}$, $\mathcal{F}_{24}$, $\mathcal{F}_{41}$, see Figure~\ref{fig:partialS_heteroclinic}.
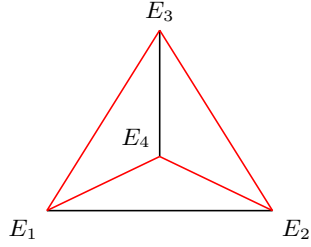
\begin{figure}
\begin{center}
\begin{tikzpicture}
\coordinate (E1)  at (0   ,0   );
\coordinate (E2)  at (3   ,0   );
\coordinate (E3)  at (1.5 ,2.4 );
\coordinate (E4)  at (1.5 ,0.72 );
\draw [semithick] (E1) -- (E2);
\draw [semithick] (E3) -- (E4);
\draw [red, semithick] (E1) -- (E4);
\draw [red, semithick] (E4) -- (E2);
\draw [red, semithick] (E2) -- (E3);
\draw [red, semithick] (E3) -- (E1);
\node [below left]  at (E1) {\small $E_1$};
\node [below right] at (E2) {\small $E_2$};
\node [above]       at (E3) {\small $E_3$};
\node [above left]  at (E4) {\small $E_4$};
\end{tikzpicture}
\end{center}
\caption{Under $\sgn c = (+,+,-,-)$, the set $\partial S$ consists of the four edges $\mathcal{F}_{13}$, $\mathcal{F}_{32}$, $\mathcal{F}_{24}$, $\mathcal{F}_{41}$.}
\label{fig:partialS_heteroclinic}
\end{figure}
There is no edge equilibrium and there is a heteroclinic cycle along $E_1$, $E_3$, $E_2$, $E_4$, $E_1$.
Indeed, on the surface $\overline{S}$, near $E_4$, the ODE~\eqref{eq:ode_A} is given by
\begin{align*}
\dot x_1 &= x_1 \left( b_4 - b_1+ f_1(x_1,x_2) \right), \\
\dot x_2 &= x_2 \left( b_4 - b_2 + f_2(x_1,x_2) \right) 
\end{align*}
with $b_4 - b_1 \geq 0$ and $b_4 - b_2 \leq 0$ (not both zero, because $b_1 < b_2$) being the eigenvalues at $E_4$ and $\lvert f_i(x_1,x_2) \rvert \to 0$ as $(x_1, x_2) \to (0,0)$. Similarly near the other corners. By \cite[Theorem 3]{hofbauer:1981}, this heteroclinic cycle is repelling if the product of the outgoing eigenvalues is larger than the product of the incoming eigenvalues along the cycle, i.e., if 
\begin{align*}
(a_3 - a_1)(b_2 - b_3)(a_2 - a_4)(b_4 - b_1) - (b_1 - b_3)(a_2 - a_3)(b_4 - b_2)(a_4 - a_1) > 0.
\end{align*}
Conversely, if $L_{\infty} < 0$ then the heteroclinic cycle is attracting.
\end{proof}

\section{Proof of Theorem \ref{thm:oppositesign}}\label{sec:proof_oppositesign}

This section is devoted to prove Theorem \ref{thm:oppositesign}. By Lemma \ref{lem:perm_implies_det>0}, permanence of the ODE~\eqref{eq:ode_exp} implies $\det J > 0$. Since $J_{11}J_{22}<0$ is assumed in Theorem \ref{thm:oppositesign}, all permanent systems fulfill $J_{12}J_{21}<0$. Thus, we are left with the four sign patterns
\begin{align*}
\begin{pmatrix}
+ & - \\ + & -
\end{pmatrix},
\begin{pmatrix}
+ & + \\ - & -
\end{pmatrix},
\begin{pmatrix}
- & - \\ + & +
\end{pmatrix},
\begin{pmatrix}
- & + \\ - & +
\end{pmatrix}
\end{align*}
for the Jacobian matrix $J$. As it is explained in \cite[Subsection 2.2]{boros:hofbauer:mueller:regensburger:2017b}, the family of ODEs~\eqref{eq:ode_exp} is invariant under the symmetry group of the square (i.e., the dihedral group $D_4$) which consists of the rotations $\mathbf{r}_0$ (by $0^\circ$), $\mathbf{r}_1$ (by $+90^\circ$), $\mathbf{r}_2$ (by $+180^\circ$), $\mathbf{r}_3$ (by $+270^\circ$) and the reflections $\mathbf{s}_0$ (along the $u$-axis), $\mathbf{s}_1$ (along the $u=v$ line), $\mathbf{s}_2$ (along the $v$-axis), $\mathbf{s}_3$ (along the $u = -v$ line). The list in \cite[Subsection 2.2]{boros:hofbauer:mueller:regensburger:2017b} reveals how the $a_i$ and $b_i$ are mapped under these eight operations. Figure~\ref{fig:plus_plus_minus_minus_patterns} illustrates how the above mentioned four sign patterns of $J$ are transformed into each other. The vector $(c_1,c_2,c_3,c_4)$ is mapped by the elements of the dihedral group $D_4$ as
\begin{align*}
\begin{aligned}
 (c_1,c_2,c_3,c_4) &\stackrel{\mathbf{r}_0}{\mapsto} (c_1,c_2,c_3,c_4),
&(c_1,c_2,c_3,c_4) &\stackrel{\mathbf{s}_0}{\mapsto} (c_1,c_2,c_4,c_3),\\
 (c_1,c_2,c_3,c_4) &\stackrel{\mathbf{r}_1}{\mapsto} (-c_4,-c_3,-c_1,-c_2),
&(c_1,c_2,c_3,c_4) &\stackrel{\mathbf{s}_1}{\mapsto} (-c_3,-c_4,-c_1,-c_2),\\
 (c_1,c_2,c_3,c_4) &\stackrel{\mathbf{r}_2}{\mapsto} (c_2,c_1,c_4,c_3),
&(c_1,c_2,c_3,c_4) &\stackrel{\mathbf{s}_2}{\mapsto} (c_2,c_1,c_3,c_4),\\
 (c_1,c_2,c_3,c_4) &\stackrel{\mathbf{r}_3}{\mapsto} (-c_3,-c_4,-c_2,-c_1),
&(c_1,c_2,c_3,c_4) &\stackrel{\mathbf{s}_3}{\mapsto} (-c_4,-c_3,-c_2,-c_1).\\
\end{aligned}
\end{align*}
Using these symmetries, once we prove case (C1) in Theorem \ref{thm:oppositesign}, the cases (C2), (C3), and (C4) follow by applying $\mathbf{s}_0$ or $\mathbf{s}_2$, $\mathbf{r}_1$ or $\mathbf{r}_3$, and $\mathbf{s}_1$ or $\mathbf{s}_3$, respectively.

Thus, from now on we mainly focus on characterizing permanence under $\sgn J = \begin{pmatrix}
+ & - \\ + & -
\end{pmatrix}$. Since we will make use of the rotation $\mathbf{r}_2$, we mention here that the ODE~\eqref{eq:ode_exp} is transformed by $\mathbf{r}_2$ into
\begin{align*}
\begin{split}
\dot u &= \e^{- a_2 u - b_2 v} - \e^{- a_1 u - b_1 v} , \\
\dot v &= \e^{- a_4 u - b_4 v} - \e^{- a_3 u - b_3 v}. 
\end{split}
\end{align*}
Further, the rotation $\mathbf{r}_2$ maps $(P_1,P_2,P_3,P_4)$ to $(-P_2,-P_1,-P_4,-P_3)$, hence, the tuple $(c_1,c_2,c_3,c_4)$ goes to the tuple $(c_2,c_1,c_4,c_3)$ (as already listed above).

\begin{figure}

\begin{center}

\begin{tikzpicture}[auto]

\node (A) at ( 4, 0) {$\begin{pmatrix} + & + \\ - & - \end{pmatrix}$};
\node (B) at ( 0, 0) {$\begin{pmatrix} + & - \\ + & - \end{pmatrix}$}; 
\node (C) at ( 0,-4) {$\begin{pmatrix} - & + \\ - & + \end{pmatrix}$};
\node (D) at ( 4,-4) {$\begin{pmatrix} - & - \\ + & + \end{pmatrix}$};

\draw[arrows={triangle 45-triangle 45}] (A) to node [above] {$\mathbf{s}_0, \mathbf{s}_2$} (B);
\draw[arrows={triangle 45-triangle 45}] (A) to node [near start] {$\mathbf{r}_1, \mathbf{r}_3$} (C);
\draw[arrows={triangle 45-triangle 45}] (A) to node [right] {$\mathbf{s}_1, \mathbf{s}_3$} (D);
\draw[arrows={triangle 45-triangle 45}] (B) to node [left] {$\mathbf{s}_1, \mathbf{s}_3$} (C);
\draw[arrows={triangle 45-triangle 45}] (B) to node [near start] {$\mathbf{r}_1, \mathbf{r}_3$} (D);
\draw[arrows={triangle 45-triangle 45}] (C) to node [below] {$\mathbf{s}_0, \mathbf{s}_2$} (D);

\tikzset{every loop/.style={min distance=8mm,in=60,out=30,looseness=8}}    
\path[arrows={-triangle 45}] (A) edge [loop above] node {$\mathbf{r}_0, \mathbf{r}_2$} ();
\tikzset{every loop/.style={min distance=8mm,in=120,out=150,looseness=8}}    
\path[arrows={-triangle 45}] (B) edge [loop above] node {$\mathbf{r}_0, \mathbf{r}_2$} ();
\tikzset{every loop/.style={min distance=8mm,in=240,out=210,looseness=8}}    
\path[arrows={-triangle 45}] (C) edge [loop above,below] node {$\mathbf{r}_0, \mathbf{r}_2$} ();
\tikzset{every loop/.style={min distance=8mm,in=300,out=-30,looseness=8}}    
\path[arrows={-triangle 45}] (D) edge [loop above,below] node {$\mathbf{r}_0, \mathbf{r}_2$} ();

\end{tikzpicture}

\end{center}
\caption{How the elements of the dihedral group $D_4$ transform the sign pattern of the Jacobian matrix $J$.}
\label{fig:plus_plus_minus_minus_patterns}
\end{figure}
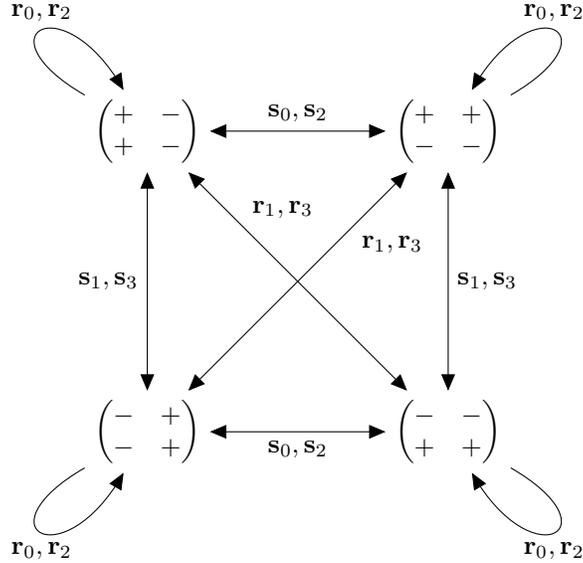

The following lemma gives a necessary condition for the ODE~\eqref{eq:ode_exp} with $\sgn J = \begin{pmatrix} + & - \\ + & - \end{pmatrix}$ to be permanent.

\begin{lemma} \label{lem:necess_for_perm}
Assume that $\sgn J = \begin{pmatrix}
+ & - \\ + & -
\end{pmatrix}$ and $\det J > 0$. Then permanence of the ODE~\eqref{eq:ode_exp} implies
\begin{enumerate}[(i)]
\item\label{it:ordering_of_as} $a_4 \leq a_2 < a_1 \leq a_3$ and
\item\label{it:c3_c4_nonpos} $c_3 \leq 0$, $c_4 \leq 0$, and $(c_3,c_4) \neq (0,0)$.
\end{enumerate}
\end{lemma}
\begin{proof}
The inequalities $a_2 < a_1$ and $a_4 < a_3$ readily follow from the assumption on the sign pattern of the Jacobian matrix. 

Next, we prove that $a_1 \leq a_3$. Assume by contradiction that $a_3 < a_1$. Then all the non-diagonal entries in the first column of $A$ in the ODE~\eqref{eq:ode_A} are negative, and therefore the corner $E_1=(1,0,0,0)$ is asymptotically stable. We claim that
\begin{align} \label{eq:E1_in_partialS}
\text{there is both a positive and a negative number among } c_2, c_3, c_4.
\end{align}
Once we show that the claim \eqref{eq:E1_in_partialS} indeed holds, it follows that $E_1 \in \partial S$. This together with the asymptotic stability of $E_1$ contradicts the permanence of the ODE~\eqref{eq:ode_exp}. Thus, $a_1 \leq a_3$ follows. We now prove the claim \eqref{eq:E1_in_partialS}. By equation \eqref{eq:c3_c4_detJ}, if one of $c_3$ and $c_4$ is positive, the other must be negative. Observe that $c_3$ and $c_4$ cannot both be zero, because then $\det J = 0$ by equation \eqref{eq:c3_c4_detJ}. We argue that each of $\sgn(c_3,c_4) = (0,-)$, $\sgn(c_3,c_4) = (-,0)$, and $\sgn(c_3,c_4) = (-,-)$ implies $c_2 > 0$. Indeed, these follow immediately from $a_4 < a_3 < a_1$ and the geometric definition \eqref{eq:c_Delta} of $c_2$, $c_3$, $c_4$, see Figure~\ref{fig:proof_of_claim_c2_pos}. The claim \eqref{eq:E1_in_partialS} is therefore proven.
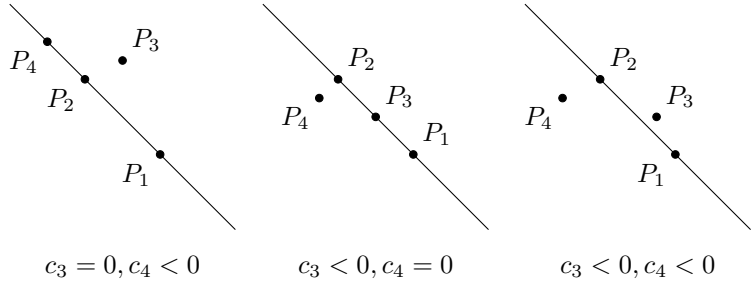
\begin{figure}
\begin{center}
\begin{tabular}{ccc}
\begin{tikzpicture}[radius=0.05]
\coordinate (P1) at (2  ,-2   );
\coordinate (P2) at (1  ,-1   );
\coordinate (P3) at (1.5,-0.75);
\coordinate (P4) at (0.5,-0.5 );
\coordinate (caption) at (1.5,-3.5);
\draw (0,0) -- (3,-3);
\draw [fill] (P1) circle; \node [below left]  at (P1) {$P_1$};
\draw [fill] (P2) circle; \node [below left]  at (P2) {$P_2$};
\draw [fill] (P3) circle; \node [above right] at (P3) {$P_3$};
\draw [fill] (P4) circle; \node [below left]  at (P4) {$P_4$};
\node at (caption) {$c_3 = 0, c_4 < 0$};
\end{tikzpicture}
&
\begin{tikzpicture}[radius=0.05]
\coordinate (P1) at (2   ,-2  );
\coordinate (P2) at (1   ,-1  );
\coordinate (P3) at (1.5,-1.5);
\coordinate (P4) at (0.75 ,-1.25);
\coordinate (caption) at (1.5,-3.5);
\draw (0,0) -- (3,-3);
\draw [fill] (P1) circle; \node [above right]  at (P1) {$P_1$};
\draw [fill] (P2) circle; \node [above right] at (P2) {$P_2$};
\draw [fill] (P3) circle; \node [above right] at (P3) {$P_3$};
\draw [fill] (P4) circle; \node [below left]  at (P4) {$P_4$};
\node at (caption) {$c_3 < 0, c_4 = 0$};
\end{tikzpicture}
&
\begin{tikzpicture}[radius=0.05]
\coordinate (P1) at (2   ,-2   );
\coordinate (P2) at (1   ,-1   );
\coordinate (P3) at (1.75,-1.5 );
\coordinate (P4) at (0.5 ,-1.25);
\coordinate (caption) at (1.5,-3.5);
\draw (0,0) -- (3,-3);
\draw [fill] (P1) circle; \node [below left]  at (P1) {$P_1$};
\draw [fill] (P2) circle; \node [above right] at (P2) {$P_2$};
\draw [fill] (P3) circle; \node [above right] at (P3) {$P_3$};
\draw [fill] (P4) circle; \node [below left]  at (P4) {$P_4$};
\node at (caption) {$c_3 < 0, c_4 < 0$};
\end{tikzpicture}
\\
\end{tabular}
\end{center}
\caption{Illustration of the proof of the claim \eqref{eq:E1_in_partialS}.}
\label{fig:proof_of_claim_c2_pos}
\end{figure}

The inequality $a_4 \leq a_2$ follows in a similar way: $a_2 < a_4$ would imply that $E_2 = (0,1,0,0)$ is asymptotically stable (and, one can show that $E_2 \in \partial S$). Alternatively, using the rotation $\mathbf{r}_2$ and the just proved fact that permanence implies $a_1 \leq a_3$, one immediately finds that permanence implies $-a_2 \leq -a_4$. Or, equivalently, $a_4 \leq a_2$. This concludes the proof of \eqref{it:ordering_of_as}.

Assume from now on that $a_4 \leq a_2 < a_1 \leq a_3$. As we already mentioned, $c_3 = c_4 = 0$ would imply that $\det J = 0$. Thus, $(c_3,c_4) \neq (0,0)$.

Next, we prove that $c_3\leq0$. Assume by contradiction that $c_3>0$. Then, by equation \eqref{eq:c3_c4_detJ}, $c_4 < 0$. Further, $c_1 < 0$ and $c_2 > 0$ follow immediately, see Figure~\ref{fig:proof_of_c1_neg_c2_pos}. Thus, $\sgn c = (-,+,+,-)$, and hence, the set $\partial S$ is the union of the four edges $\mathcal{F}_{12}$, $\mathcal{F}_{24}$, $\mathcal{F}_{43}$, $\mathcal{F}_{31}$, see the left panel in Figure~\ref{fig:partialS_c3_pos}. As can be read from Figure~\ref{fig:proof_of_c1_neg_c2_pos}, we have $b_2 < b_4$.

If $a_4 < a_2$ also holds, there exists an equilibrium $E_{24}$ on the edge $\mathcal{F}_{24}$, see the middle panel in Figure~\ref{fig:partialS_c3_pos}. By the equations \eqref{eq:Gamma_special} in Appendix \ref{sec:app_special},
\begin{align*}
\sgn \Gamma_{24}^1 &= -\sgn c_3 = -1, \\
\sgn \Gamma_{24}^3 &= +\sgn c_1 = -1,
\end{align*}
i.e., both of the external eigenvalues at $E_{24}$ are negative. Thus, $E_{24}$ is asymptotically stable, contradicting that the flow on $S$ is permanent.

To obtain a contradiction, it is left to show that the flow on $S$ is not permanent when $a_4 = a_2$. The eigenvalues at $E_2$ in the directions $E_1$, $E_3$, and $E_4$ are negative, negative, and zero, respectively. Further, the flow on the edge $\mathcal{F}_{24}$ goes from $E_4$ to $E_2$. See the right panel in Figure~\ref{fig:partialS_c3_pos}. 

Therefore, the stable manifold at $E_2$ is $2$-dimensional (and is contained in the facet $\mathcal{F}_{123}$), the center manifold at $E_2$ is $1$-dimensional (and is contained in the edge $\mathcal{F}_{24}$), and since the flow on the edge $\mathcal{F}_{24}$ goes from $E_4$ to $E_2$, $E_2$ is attracting on the center manifold. By the reduction principle (see~\cite[Theorem~5.2]{kuznetsov:2004}), $E_2$ is asymptotically stable in $\Delta_4$, contradicting that the flow on $S$ is permanent. Therefore, we conclude that permanence implies $c_3 \leq 0$.

\begin{figure}
\begin{center}
\begin{tikzpicture}[radius=0.05]
\coordinate (P1) at (2  ,-2   );
\coordinate (P2) at (1  ,-1   );
\coordinate (P3) at (2.5,-1.5 );
\coordinate (P4) at (0.5,-0.25);
\draw (0,0) -- (3,-3);
\draw (1,0) -- (1,-3);
\draw (2,0) -- (2,-3);
\draw [fill] (P1) circle; \node [below left] at (P1) {$P_1$};
\draw [fill] (P2) circle; \node [below left] at (P2) {$P_2$};
\draw [fill] (P3) circle; \node [right]      at (P3) {$P_3$};
\draw [fill] (P4) circle; \node [above]      at (P4) {$P_4$};
\end{tikzpicture}
\end{center}
\caption{Illustration of the proof that $c_3 > 0$ implies $c_1 < 0$ and $c_2 > 0$.}
\label{fig:proof_of_c1_neg_c2_pos}
\end{figure}
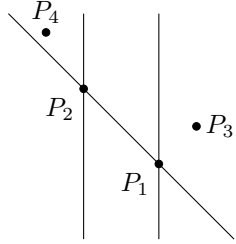

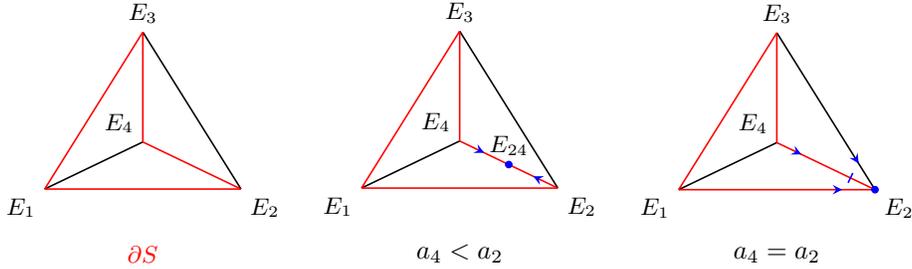
\begin{figure}
\begin{center}
\begin{tabular}{ccc}
\begin{tikzpicture}[scale=.87]
\coordinate (E1)  at (0   ,0   );
\coordinate (E2)  at (3   ,0   );
\coordinate (E3)  at (1.5 ,2.4 );
\coordinate (E4)  at (1.5 ,0.72 );
\coordinate (caption) at (1.5,-1);
\draw [semithick] (E1) -- (E4);
\draw [semithick] (E2) -- (E3);
\draw [red, semithick] (E1) -- (E2);
\draw [red, semithick] (E2) -- (E4);
\draw [red, semithick] (E4) -- (E3);
\draw [red, semithick] (E3) -- (E1);
\node [below left]  at (E1) {\small $E_1$};
\node [below right] at (E2) {\small $E_2$};
\node [above]       at (E3) {\small $E_3$};
\node [above left]  at (E4) {\small $E_4$};
\node at (caption) {\small \textcolor{red}{$\partial S$}};
\end{tikzpicture}
&
\begin{tikzpicture}[scale=.87]
\coordinate (E1)  at (0   ,0   );
\coordinate (E2)  at (3   ,0   );
\coordinate (E3)  at (1.5 ,2.4 );
\coordinate (E4)  at (1.5 ,0.72 );
\coordinate (E24) at (2.25,0.36);
\coordinate (arr42) at (1.875,0.54);
\coordinate (arr24) at (2.625,0.18);
\coordinate (caption) at (1.5,-1);
\draw [semithick] (E1) -- (E4);
\draw [semithick] (E2) -- (E3);
\draw [red, semithick] (E1) -- (E2);
\draw [red, semithick] (E2) -- (E4);
\draw [red, semithick] (E4) -- (E3);
\draw [red, semithick] (E3) -- (E1);
\node [below left]  at (E1) {\small $E_1$};
\node [below right] at (E2) {\small $E_2$};
\node [above]       at (E3) {\small $E_3$};
\node [above left]  at (E4) {\small $E_4$};
\draw [fill,blue] (E24) circle [radius=0.05];
\node [above] at (E24) {\small $E_{24}$};
\draw [blue,draw opacity=0,semithick,-stealth] (E4) -- (arr42);
\draw [blue,draw opacity=0,semithick,-stealth] (E2) -- (arr24);
\node at (caption) {$a_4 < a_2$};
\end{tikzpicture}
&
\begin{tikzpicture}[scale=.87]
\coordinate (E1)  at (0   ,0   );
\coordinate (E2)  at (3   ,0   );
\coordinate (E3)  at (1.5 ,2.4 );
\coordinate (E4)  at (1.5 ,0.72);
\coordinate (arr42) at (1.875,0.54);
\coordinate (arr24a) at (2.625+.036,0.18+.075);
\coordinate (arr24b) at (2.625-.036,0.18-.075);
\coordinate (arr21) at (2.5,0);
\coordinate (arr23) at (2.75,0.4);
\coordinate (caption) at (1.5,-1);
\draw [semithick] (E1) -- (E4);
\draw [semithick] (E2) -- (E3);
\draw [red, semithick] (E1) -- (E2);
\draw [red, semithick] (E2) -- (E4);
\draw [red, semithick] (E4) -- (E3);
\draw [red, semithick] (E3) -- (E1);
\node [below left]  at (E1) {\small $E_1$};
\node [below right] at (E2) {\small $E_2$};
\node [above]       at (E3) {\small $E_3$};
\node [above left]  at (E4) {\small $E_4$};
\draw [fill,blue] (E2) circle [radius=0.05];
\draw [blue,draw opacity=0,semithick,-stealth] (E4) -- (arr42);
\draw [blue,draw opacity=0,semithick,-stealth] (E1) -- (arr21);
\draw [blue,draw opacity=0,semithick,-stealth] (E3) -- (arr23);
\draw [blue,semithick,-] (arr24a) -- (arr24b);
\node at (caption) {$a_4 = a_2$};
\end{tikzpicture}
\\
\end{tabular}
\end{center}
\caption{Under $\sgn c = (-,+,+,-)$, the set $\partial S$ (left panel), the asymptotically stable equilibrium $E_{24}$ (middle panel), and the asymptotically stable equilibrium $E_2$ (right panel).}
\label{fig:partialS_c3_pos}
\end{figure}

Finally, proving that $c_4$ is also non-positive can be done in a similar way, but it is more elegant to apply the rotation $\mathbf{r}_2$: it maps the tuple $(c_1,c_2,c_3,c_4)$ to the tuple $(c_2,c_1,c_4,c_3)$. This concludes the proof of \eqref{it:c3_c4_nonpos}.
\end{proof}

The following lemma reveals further connections between the signs of $c_1$, $c_2$, $c_3$, $c_4$.

\begin{lemma} \label{lem:c3_or_c4_zero_implies_sgn_c1_and_c2}
Assume that $\sgn J = \begin{pmatrix}
+ & - \\ + & -
\end{pmatrix}$, $\det J > 0$, and $a_4 \leq a_2 < a_1 \leq a_3$. Then
\begin{enumerate}[(i)]
\item\label{it:c3_zero_c4_neg} $c_3 = 0$ and $c_4 < 0$ imply $c_1 \leq 0$ and $c_2 > 0$,
\item\label{it:c3_neg_c4_zero} $c_3 < 0$ and $c_4 = 0$ imply $c_1 > 0$ and $c_2 \leq 0$.
\end{enumerate}
\end{lemma}
\begin{proof}
Under $c_3 = 0$ and $c_4 < 0$, the configuration of the points $P_1$, $P_2$, $P_3$, $P_4$ is as shown in Figure~\ref{fig:c3_zero_c4_neg}. Using \eqref{eq:c_Delta}, one immediately obtains $c_1 \leq 0$ (with equality if and only if $a_2 = a_4$ (or, equivalently, $P_2 = P_4$)) and $c_2 > 0$. This concludes the proof of \eqref{it:c3_zero_c4_neg}.

One can prove \eqref{it:c3_neg_c4_zero} in a similar way. Alternatively, one may apply the rotation $\mathbf{r}_2$, it maps the tuple $(c_1,c_2,c_3,c_4)$ to the tuple $(c_2,c_1,c_4,c_3)$. Thus, the statement \eqref{it:c3_neg_c4_zero} follows from \eqref{it:c3_zero_c4_neg}.

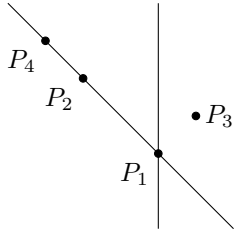
\begin{figure}
\begin{center}
\begin{tikzpicture}[radius=0.05]
\coordinate (P1) at (2  ,-2   );
\coordinate (P2) at (1  ,-1   );
\coordinate (P3) at (2.5,-1.5 );
\coordinate (P4) at (0.5,-0.5);
\draw (0,0) -- (3,-3);
\draw (2,0) -- (2,-3);
\draw [fill] (P1) circle; \node [below left] at (P1) {$P_1$};
\draw [fill] (P2) circle; \node [below left] at (P2) {$P_2$};
\draw [fill] (P3) circle; \node [right]      at (P3) {$P_3$};
\draw [fill] (P4) circle; \node [below left] at (P4) {$P_4$};
\end{tikzpicture}
\end{center}
\caption{Illustration of the proof that $c_3 = 0$ and $c_4 < 0$ imply $c_1 \leq 0$ and $c_2 > 0$.}
\label{fig:c3_zero_c4_neg}
\end{figure}
\end{proof}

By the inequality \eqref{eq:c1_c2_detJ} and Lemmata \ref{lem:necess_for_perm} and \ref{lem:c3_or_c4_zero_implies_sgn_c1_and_c2}, there are only $9$ possible sign patterns of $c$ that a permanent ODE~\eqref{eq:ode_exp} with $\sgn J = \begin{pmatrix}
+ & - \\ + & -
\end{pmatrix}$ can lead to, see Figure~\ref{fig:9_cs}. See Figure~\ref{fig:9_P1P2P3P4} for the relative positions of the four points $P_1$, $P_2$, $P_3$, $P_4$ in the $9$ cases.

If $\sgn c = (+,+,-,-)$ then the surface $S$ is given by $\frac{x_1^{|c_1|}x_2^{|c_2|}}{x_3^{|c_3|}x_4^{|c_4|}}=1$ and thus, $\partial S$ consists of the four edges $\mathcal{F}_{14}$, $\mathcal{F}_{42}$, $\mathcal{F}_{23}$, $\mathcal{F}_{31}$. The four points $P_1$, $P_2$, $P_3$, $P_4$ form a quadrangle with diagonals $P_1P_2$ and $P_3P_4$.

If $\sgn c = (+,-,-,-)$ then the surface $S$ is given by $\frac{x_1^{|c_1|}}{x_2^{|c_2|}x_3^{|c_3|}x_4^{|c_4|}}=1$ and thus, $\partial S$ consists of the three edges $\mathcal{F}_{23}$, $\mathcal{F}_{34}$, $\mathcal{F}_{42}$. The three points $P_2$, $P_3$, $P_4$ form a triangle, and $P_1$ lies inside.

If $\sgn c = (+,0,-,-)$ then the surface $S$ is given by $\frac{x_1^{|c_1|}}{x_3^{|c_3|}x_4^{|c_4|}}=1$ and thus, $\partial S$ consists of the two edges $\mathcal{F}_{42}$, $\mathcal{F}_{23}$ and the curve
\begin{align*}
\mathcal{C}_{34}^1 = \{x \in \Delta_4~|~ x_2 = 0 \text{ and } x_1^{|c_1|} = x_3^{|c_3|}x_4^{|c_4|}\} \subseteq \mathcal{F}_{134},
\end{align*}
which connects $E_3$ and $E_4$. The three points $P_2$, $P_3$, $P_4$ form a triangle, and $P_1$ lies inside the edge $P_3P_4$.

If $\sgn c = (-,+,-,-)$ then the surface $S$ is given by $\frac{x_2^{|c_2|}}{x_1^{|c_1|}x_3^{|c_3|}x_4^{|c_4|}}=1$ and thus, $\partial S$ consists of the three edges $\mathcal{F}_{13}$, $\mathcal{F}_{34}$, $\mathcal{F}_{41}$. The three points $P_1$, $P_3$, $P_4$ form a triangle, and $P_2$ lies inside.

If $\sgn c = (0,+,-,-)$ then the surface $S$ is given by $\frac{x_2^{|c_2|}}{x_3^{|c_3|}x_4^{|c_4|}}=1$ and thus, $\partial S$ consists of the two edges $\mathcal{F}_{41}$, $\mathcal{F}_{13}$ and the curve
\begin{align*}
\mathcal{C}_{34}^2 = \{x \in \Delta_4~|~ x_1 = 0 \text{ and } x_2^{|c_2|} = x_3^{|c_3|}x_4^{|c_4|}\} \subseteq \mathcal{F}_{234},
\end{align*}
which connects $E_3$ and $E_4$. The three points $P_1$, $P_3$, $P_4$ form a triangle, and $P_2$ lies inside the edge $P_3P_4$.

If $\sgn c = (+,-,-,0)$ then the surface $S$ is given by $\frac{x_1^{|c_1|}}{x_2^{|c_2|}x_3^{|c_3|}}=1$ and thus, $\partial S$ consists of the two edges $\mathcal{F}_{34}$, $\mathcal{F}_{42}$ and the curve
\begin{align*}
\mathcal{C}_{23}^1 = \{x \in \Delta_4~|~ x_4 = 0 \text{ and } x_1^{|c_1|} = x_2^{|c_2|}x_3^{|c_3|}\} \subseteq \mathcal{F}_{123},
\end{align*}
which connects $E_2$ and $E_3$. The three points $P_2$, $P_3$, $P_4$ form a triangle, and $P_1$ lies inside the edge $P_2P_3$.

If $\sgn c = (+,0,-,0)$ then the surface $S$ is given by $x_1=x_3$, i.e., $S$ is the triangle with vertices $E_2$, $E_4$, $E_m = \frac12 E_1 + \frac12 E_3$. The three points $P_2$, $P_3$, $P_4$ form a triangle, and $P_1$ coincides with $P_3$.

If $\sgn c = (-,+,0,-)$ then the surface $S$ is given by $\frac{x_2^{|c_2|}}{x_1^{|c_1|}x_4^{|c_4|}}=1$ and thus, $\partial S$ consists of the two edges $\mathcal{F}_{43}$, $\mathcal{F}_{31}$ and the curve
\begin{align*}
\mathcal{C}_{14}^2 = \{x \in \Delta_4~|~ x_3 = 0 \text{ and } x_2^{|c_2|} = x_1^{|c_1|}x_4^{|c_4|}\} \subseteq \mathcal{F}_{124},
\end{align*}
which connects $E_1$ and $E_4$. The three points $P_1$, $P_3$, $P_4$ form a triangle, and $P_2$ lies inside the edge $P_1P_4$.

If $\sgn c = (0,+,0,-)$ then the surface $S$ is given by $x_2=x_4$, i.e., $S$ is the triangle with vertices $E_1$, $E_3$, $E_m = \frac12 E_2 + \frac12 E_4$. The three points $P_1$, $P_3$, $P_4$ form a triangle, and $P_2$ coincides with $P_4$.

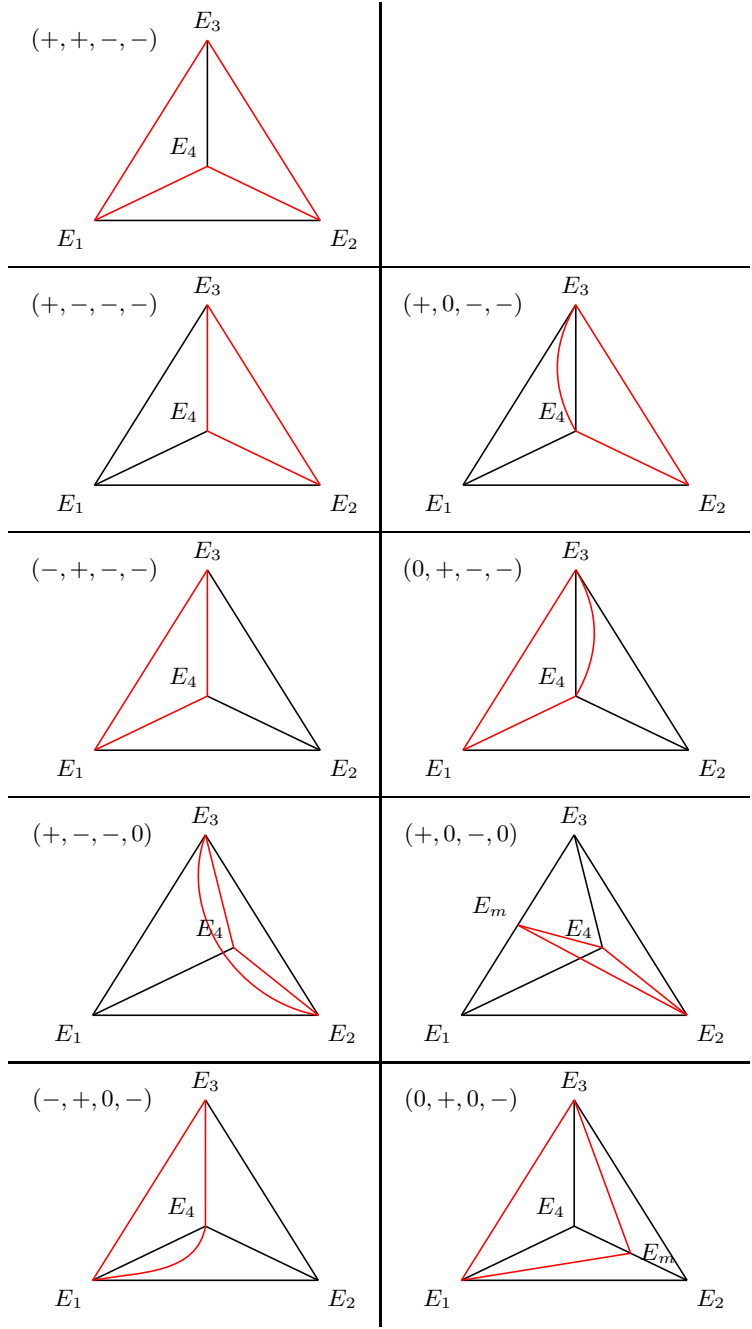
\begin{figure}
\captionsetup{singlelinecheck=false}
\begin{center}
\begin{tabular}{c|c}
\begin{tikzpicture}
\coordinate (E1)  at (0   ,0   );
\coordinate (E2)  at (3   ,0   );
\coordinate (E3)  at (1.5 ,2.4 );
\coordinate (E4)  at (1.5 ,0.72 );
\coordinate (caption) at (0,2.4);
\draw [semithick] (E1) -- (E2);
\draw [semithick] (E3) -- (E4);
\draw [red, semithick] (E1) -- (E4);
\draw [red, semithick] (E4) -- (E2);
\draw [red, semithick] (E2) -- (E3);
\draw [red, semithick] (E3) -- (E1);
\node [below left]  at (E1) {\small $E_1$};
\node [below right] at (E2) {\small $E_2$};
\node [above]       at (E3) {\small $E_3$};
\node [above left]  at (E4) {\small $E_4$};
\node at (caption) {\small $(+,+,-,-)$};
\end{tikzpicture}
&
\\
\hline
\begin{tikzpicture}
\coordinate (E1)  at (0   ,0   );
\coordinate (E2)  at (3   ,0   );
\coordinate (E3)  at (1.5 ,2.4 );
\coordinate (E4)  at (1.5 ,0.72 );
\coordinate (caption) at (0,2.4);
\draw [semithick] (E1) -- (E2);
\draw [semithick] (E1) -- (E3);
\draw [semithick] (E1) -- (E4);
\draw [red, semithick] (E4) -- (E2);
\draw [red, semithick] (E2) -- (E3);
\draw [red, semithick] (E3) -- (E4);
\node [below left]  at (E1) {\small $E_1$};
\node [below right] at (E2) {\small $E_2$};
\node [above]       at (E3) {\small $E_3$};
\node [above left]  at (E4) {\small $E_4$};
\node at (caption) {\small $(+,-,-,-)$};
\end{tikzpicture}
&
\begin{tikzpicture}
\coordinate (E1)  at (0   ,0   );
\coordinate (E2)  at (3   ,0   );
\coordinate (E3)  at (1.5 ,2.4 );
\coordinate (E4)  at (1.5 ,0.72 );
\coordinate (caption) at (0,2.4);
\draw [semithick] (E1) -- (E2);
\draw [semithick] (E1) -- (E3);
\draw [semithick] (E1) -- (E4);
\draw [semithick] (E3) -- (E4);
\draw [red, semithick] (E4) -- (E2);
\draw [red, semithick] (E2) -- (E3);
\draw [red, semithick] (E3) to [out=-120,in=120] (E4);
\node [below left]  at (E1) {\small $E_1$};
\node [below right] at (E2) {\small $E_2$};
\node [above]       at (E3) {\small $E_3$};
\node [above left]  at (E4) {\small $E_4$};
\node at (caption) {\small $(+,0,-,-)$};
\end{tikzpicture}
\\
\hline
\begin{tikzpicture}
\coordinate (E1)  at (0   ,0   );
\coordinate (E2)  at (3   ,0   );
\coordinate (E3)  at (1.5 ,2.4 );
\coordinate (E4)  at (1.5 ,0.72 );
\coordinate (caption) at (0,2.4);
\draw [semithick] (E2) -- (E1);
\draw [semithick] (E2) -- (E3);
\draw [semithick] (E2) -- (E4);
\draw [red, semithick] (E1) -- (E4);
\draw [red, semithick] (E4) -- (E3);
\draw [red, semithick] (E3) -- (E1);
\node [below left]  at (E1) {\small $E_1$};
\node [below right] at (E2) {\small $E_2$};
\node [above]       at (E3) {\small $E_3$};
\node [above left]  at (E4) {\small $E_4$};
\node at (caption) {\small $(-,+,-,-)$};
\end{tikzpicture}
&
\begin{tikzpicture}
\coordinate (E1)  at (0   ,0   );
\coordinate (E2)  at (3   ,0   );
\coordinate (E3)  at (1.5 ,2.4 );
\coordinate (E4)  at (1.5 ,0.72 );
\coordinate (caption) at (0,2.4);
\draw [semithick] (E2) -- (E1);
\draw [semithick] (E2) -- (E3);
\draw [semithick] (E2) -- (E4);
\draw [semithick] (E3) -- (E4);
\draw [red, semithick] (E1) -- (E3);
\draw [red, semithick] (E1) -- (E4);
\draw [red, semithick] (E3) to [out=-60,in=60] (E4);
\node [below left]  at (E1) {\small $E_1$};
\node [below right] at (E2) {\small $E_2$};
\node [above]       at (E3) {\small $E_3$};
\node [above left]  at (E4) {\small $E_4$};
\node at (caption) {\small $(0,+,-,-)$};
\end{tikzpicture}
\\
\hline
\begin{tikzpicture}
\coordinate (E1)  at (0   ,0   );
\coordinate (E2)  at (3   ,0   );
\coordinate (E3)  at (1.5 ,2.4 );
\coordinate (E4)  at (1.5*1.25 ,0.72*1.25);
\coordinate (caption) at (0,2.4);
\draw [semithick] (E1) -- (E2);
\draw [semithick] (E1) -- (E3);
\draw [semithick] (E1) -- (E4);
\draw [semithick] (E2) -- (E3);
\draw [red, semithick] (E4) -- (E2);
\draw [red, semithick] (E2) to [out=170,in=-110] (E3);
\draw [red, semithick] (E3) -- (E4);
\node [below left]  at (E1) {\small $E_1$};
\node [below right] at (E2) {\small $E_2$};
\node [above]       at (E3) {\small $E_3$};
\node [above left]  at (E4) {\small $E_4$};
\node at (caption) {\small $(+,-,-,0)$};
\end{tikzpicture}
&
\begin{tikzpicture}
\coordinate (E1)  at (0   ,0   );
\coordinate (E2)  at (3   ,0   );
\coordinate (E3)  at (1.5 ,2.4 );
\coordinate (E4)  at (1.5*1.25 ,0.72*1.25 );
\coordinate (Em)  at (0.75,1.2);
\coordinate (caption) at (0,2.4);
\draw [semithick] (E1) -- (E2);
\draw [semithick] (E1) -- (E3);
\draw [semithick] (E1) -- (E4);
\draw [semithick] (E2) -- (E3);
\draw [semithick] (E3) -- (E4);
\draw [red, semithick] (E4) -- (E2);
\draw [red, semithick] (E2) -- (Em);
\draw [red, semithick] (Em) -- (E4);
\node [below left]  at (E1) {\small $E_1$};
\node [below right] at (E2) {\small $E_2$};
\node [above]       at (E3) {\small $E_3$};
\node [above left]  at (E4) {\small $E_4$};
\node [above left]  at (Em) {\small $E_m$};
\node at (caption) {\small $(+,0,-,0)$};
\end{tikzpicture}
\\
\hline
\begin{tikzpicture}
\coordinate (E1)  at (0   ,0   );
\coordinate (E2)  at (3   ,0   );
\coordinate (E3)  at (1.5 ,2.4 );
\coordinate (E4)  at (1.5 ,0.72 );
\coordinate (caption) at (0,2.4);
\draw [semithick] (E2) -- (E1);
\draw [semithick] (E2) -- (E3);
\draw [semithick] (E2) -- (E4);
\draw [semithick] (E1) -- (E4);
\draw [red, semithick] (E1) to [out=10,in=-100] (E4);
\draw [red, semithick] (E4) -- (E3);
\draw [red, semithick] (E3) -- (E1);
\node [below left]  at (E1) {\small $E_1$};
\node [below right] at (E2) {\small $E_2$};
\node [above]       at (E3) {\small $E_3$};
\node [above left]  at (E4) {\small $E_4$};
\node at (caption) {\small $(-,+,0,-)$};
\end{tikzpicture}
&
\begin{tikzpicture}
\coordinate (E1) at (0   ,0   );
\coordinate (E2) at (3   ,0   );
\coordinate (E3) at (1.5 ,2.4 );
\coordinate (E4) at (1.5 ,0.72);
\coordinate (Em) at (2.25,0.36);
\coordinate (caption) at (0,2.4);
\draw [semithick] (E2) -- (E1);
\draw [semithick] (E2) -- (E3);
\draw [semithick] (E2) -- (E4);
\draw [semithick] (E3) -- (E4);
\draw [semithick] (E1) -- (E4);
\draw [red, semithick] (E1) -- (E3);
\draw [red, semithick] (E1) -- (Em);
\draw [red, semithick] (E3) -- (Em);
\node [below left]  at (E1) {\small $E_1$};
\node [below right] at (E2) {\small $E_2$};
\node [above]       at (E3) {\small $E_3$};
\node [above left]  at (E4) {\small $E_4$};
\node [right] at (Em) {\small $E_m$};
\node at (caption) {\small $(0,+,0,-)$};
\end{tikzpicture}
\\
\end{tabular}
\end{center}
\caption[]{The $9$ sign patterns of $c$ and the corresponding $\partial S$ (red) that a permanent ODE~\eqref{eq:ode_exp} with $\sgn J = \begin{pmatrix} + & - \\ + & - \end{pmatrix}$ can lead to.}
\label{fig:9_cs}
\end{figure}

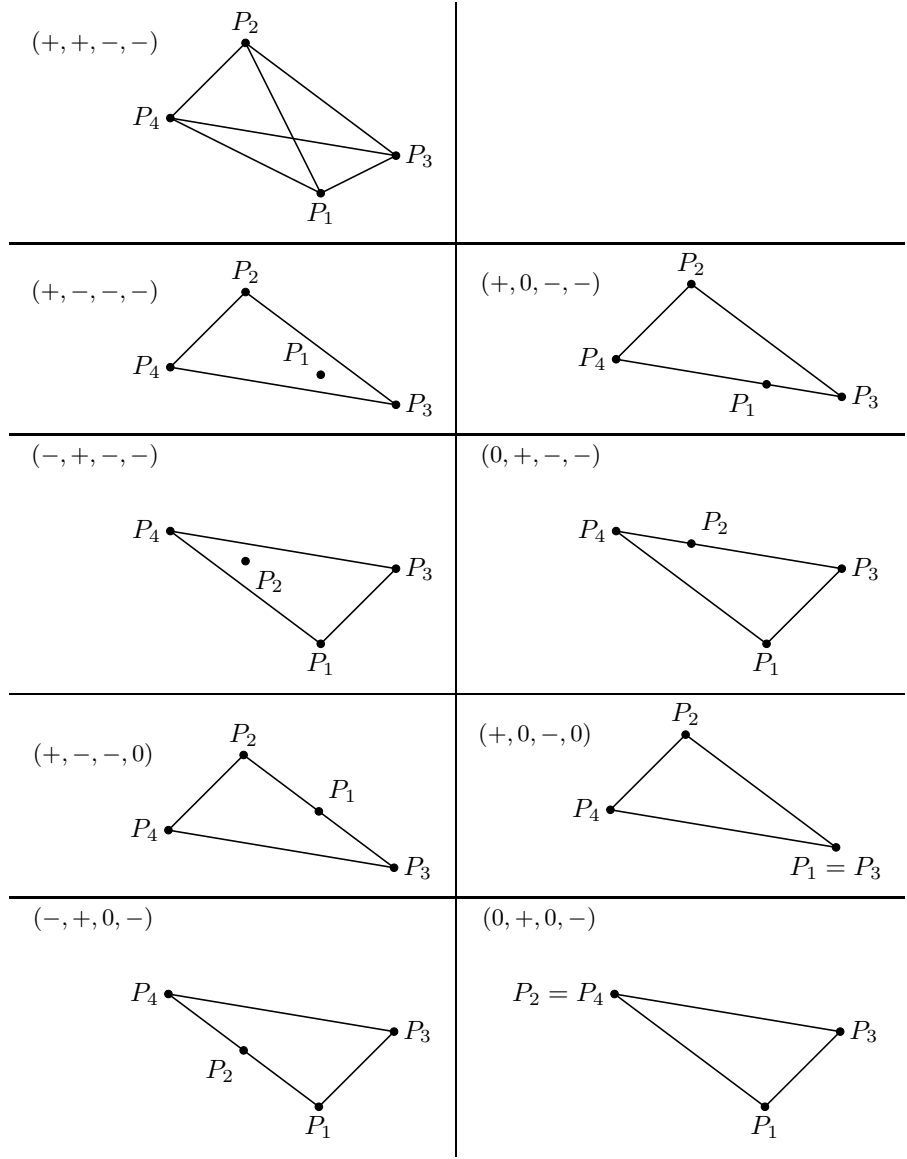
\begin{figure}
\begin{center}
\begin{tabular}{c|c}
\begin{tikzpicture}[radius=0.05]
\coordinate (P1)  at (2  ,-2   );
\coordinate (P2)  at (1  ,0   );
\coordinate (P3)  at (3 ,-1.5 );
\coordinate (P4)  at (0 ,-1);
\coordinate (caption) at (-1,0);
\draw [semithick] (P1) -- (P2);
\draw [semithick] (P1) -- (P3);
\draw [semithick] (P1) -- (P4);
\draw [semithick] (P2) -- (P3);
\draw [semithick] (P2) -- (P4);
\draw [semithick] (P3) -- (P4);
\draw [fill] (P1) circle; \node [below] at (P1) {$P_1$};
\draw [fill] (P2) circle; \node [above] at (P2) {$P_2$};
\draw [fill] (P3) circle; \node [right] at (P3) {$P_3$};
\draw [fill] (P4) circle; \node [left]  at (P4) {$P_4$};
\node at (caption) {\small $(+,+,-,-)$};
\end{tikzpicture}
&
\\
\hline
\begin{tikzpicture}[radius=0.05]
\coordinate (P1)  at (2  ,-1.1   );
\coordinate (P2)  at (1  ,0   );
\coordinate (P3)  at (3 ,-1.5 );
\coordinate (P4)  at (0 ,-1);
\coordinate (caption) at (-1,0);
\draw [semithick] (P2) -- (P3);
\draw [semithick] (P2) -- (P4);
\draw [semithick] (P3) -- (P4);
\draw [fill] (P1) circle; \node [above left] at (P1) {$P_1$};
\draw [fill] (P2) circle; \node [above] at (P2) {$P_2$};
\draw [fill] (P3) circle; \node [right] at (P3) {$P_3$};
\draw [fill] (P4) circle; \node [left]  at (P4) {$P_4$};
\node at (caption) {\small $(+,-,-,-)$};
\end{tikzpicture}
&
\begin{tikzpicture}[radius=0.05]
\coordinate (P1)  at (2  ,-1.33333   );
\coordinate (P2)  at (1  ,0   );
\coordinate (P3)  at (3 ,-1.5 );
\coordinate (P4)  at (0 ,-1);
\coordinate (caption) at (-1,0);
\draw [semithick] (P2) -- (P3);
\draw [semithick] (P2) -- (P4);
\draw [semithick] (P3) -- (P4);
\draw [fill] (P1) circle; \node [below left] at (P1) {$P_1$};
\draw [fill] (P2) circle; \node [above] at (P2) {$P_2$};
\draw [fill] (P3) circle; \node [right] at (P3) {$P_3$};
\draw [fill] (P4) circle; \node [left]  at (P4) {$P_4$};
\node at (caption) {\small $(+,0,-,-)$};
\end{tikzpicture}
\\
\hline
\begin{tikzpicture}[radius=0.05]
\coordinate (P1)  at (2  ,-2.5   );
\coordinate (P2)  at (1  ,-1.4   );
\coordinate (P3)  at (3 ,-1.5 );
\coordinate (P4)  at (0 ,-1);
\coordinate (caption) at (-1,0);
\draw [semithick] (P1) -- (P3);
\draw [semithick] (P1) -- (P4);
\draw [semithick] (P3) -- (P4);
\draw [fill] (P1) circle; \node [below] at (P1) {$P_1$};
\draw [fill] (P2) circle; \node [below right] at (P2) {$P_2$};
\draw [fill] (P3) circle; \node [right] at (P3) {$P_3$};
\draw [fill] (P4) circle; \node [left]  at (P4) {$P_4$};
\node at (caption) {\small $(-,+,-,-)$};
\end{tikzpicture}
&
\begin{tikzpicture}[radius=0.05]
\coordinate (P1)  at (2  ,-2.5  );
\coordinate (P2)  at (1  ,-1.16666   );
\coordinate (P3)  at (3 ,-1.5 );
\coordinate (P4)  at (0 ,-1);
\coordinate (caption) at (-1,0);
\draw [semithick] (P1) -- (P3);
\draw [semithick] (P1) -- (P4);
\draw [semithick] (P3) -- (P4);
\draw [fill] (P1) circle; \node [below] at (P1) {$P_1$};
\draw [fill] (P2) circle; \node [above right] at (P2) {$P_2$};
\draw [fill] (P3) circle; \node [right] at (P3) {$P_3$};
\draw [fill] (P4) circle; \node [left]  at (P4) {$P_4$};
\node at (caption) {\small $(0,+,-,-)$};
\end{tikzpicture}
\\
\hline
\begin{tikzpicture}[radius=0.05]
\coordinate (P1)  at (2  ,-0.75);
\coordinate (P2)  at (1  ,0   );
\coordinate (P3)  at (3 ,-1.5 );
\coordinate (P4)  at (0 ,-1);
\coordinate (caption) at (-1,0);
\draw [semithick] (P2) -- (P3);
\draw [semithick] (P2) -- (P4);
\draw [semithick] (P3) -- (P4);
\draw [fill] (P1) circle; \node [above right] at (P1) {$P_1$};
\draw [fill] (P2) circle; \node [above] at (P2) {$P_2$};
\draw [fill] (P3) circle; \node [right] at (P3) {$P_3$};
\draw [fill] (P4) circle; \node [left]  at (P4) {$P_4$};
\node at (caption) {\small $(+,-,-,0)$};
\end{tikzpicture}
&
\begin{tikzpicture}[radius=0.05]
\coordinate (P2)  at (1  ,0   );
\coordinate (P3)  at (3 ,-1.5 );
\coordinate (P4)  at (0 ,-1);
\coordinate (caption) at (-1,0);
\draw [semithick] (P2) -- (P3);
\draw [semithick] (P2) -- (P4);
\draw [semithick] (P3) -- (P4);
\draw [fill] (P2) circle; \node [above] at (P2) {$P_2$};
\draw [fill] (P3) circle; \node [below] at (P3) {$P_1=P_3$};
\draw [fill] (P4) circle; \node [left]  at (P4) {$P_4$};
\node at (caption) {\small $(+,0,-,0)$};
\end{tikzpicture}
\\
\hline
\begin{tikzpicture}[radius=0.05]
\coordinate (P1)  at (2  ,-2.5   );
\coordinate (P2)  at (1  ,-1.75   );
\coordinate (P3)  at (3 ,-1.5 );
\coordinate (P4)  at (0 ,-1);
\coordinate (caption) at (-1,0);
\draw [semithick] (P1) -- (P3);
\draw [semithick] (P1) -- (P4);
\draw [semithick] (P3) -- (P4);
\draw [fill] (P1) circle; \node [below] at (P1) {$P_1$};
\draw [fill] (P2) circle; \node [below left] at (P2) {$P_2$};
\draw [fill] (P3) circle; \node [right] at (P3) {$P_3$};
\draw [fill] (P4) circle; \node [left]  at (P4) {$P_4$};
\node at (caption) {\small $(-,+,0,-)$};
\end{tikzpicture}
&
\begin{tikzpicture}[radius=0.05]
\coordinate (P1)  at (2  ,-2.5  );
\coordinate (P3)  at (3 ,-1.5 );
\coordinate (P4)  at (0 ,-1);
\coordinate (caption) at (-1,0);
\draw [semithick] (P1) -- (P3);
\draw [semithick] (P1) -- (P4);
\draw [semithick] (P3) -- (P4);
\draw [fill] (P1) circle; \node [below] at (P1) {$P_1$};
\draw [fill] (P3) circle; \node [right] at (P3) {$P_3$};
\draw [fill] (P4) circle; \node [left]  at (P4) {$P_2=P_4$};
\node at (caption) {\small $(0,+,0,-)$};
\end{tikzpicture}
\\
\end{tabular}
\end{center}
\caption[]{The $9$ sign patterns of $c$ and the corresponding relative positions of the four points $P_1$, $P_2$, $P_3$, $P_4$ that a permanent ODE~\eqref{eq:ode_exp} with $\sgn J = \begin{pmatrix}
+ & - \\ + & -
\end{pmatrix}$ can lead to.}
\label{fig:9_P1P2P3P4}
\end{figure}

To prove Theorem \ref{thm:oppositesign} (C1), our task is to investigate permanence of the ODE~\eqref{eq:ode_exp} under the $9$ sign patterns of $c$ listed in Figure~\ref{fig:9_cs}. Due to the next lemma, none of the corners of the simplex $\Delta_4$ can attract an orbit from $S$.

\begin{lemma}\label{lem:corners_doesnt_attract}
Assume that $\sgn J = \begin{pmatrix}
+ & - \\ + & -
\end{pmatrix}$, $\det J > 0$, and $a_4 \leq a_2 < a_1 \leq a_3$. Consider any of the $9$ cases in Figure~\ref{fig:9_cs}. Then, for each $i \in \{1,2,3,4\}$, if $E_i \in \partial S$, it cannot attract an orbit from $S$.
\end{lemma}
\begin{proof}
If $a_4 < a_2$ then the eigenvalue at $E_2$ in the direction $E_4$ is positive. Thus, by Lemma \ref{lem:saturated}, $E_2$ cannot attract an orbit from $\Delta_4^\circ$, and hence, from $S$.

If $a_4 = a_2$ and $E_2 \in \partial S$ then $c_3 < 0$ (see Figure~\ref{fig:9_cs}). Since $c_3 = \Delta(142)$ (see \eqref{eq:c_Delta}), we obtain $b_4 < b_2$ (see the left panel in Figure~\ref{fig:no_approach_E2}). The eigenvalues at $E_2$ in the directions $E_1$, $E_3$, and $E_4$ are negative, negative, and zero, respectively. Since $b_4 < b_2$, the flow on the edge $\mathcal{F}_{24}$ goes from $E_2$ to $E_4$. See the right panel in Figure~\ref{fig:no_approach_E2}. Therefore, the stable manifold at $E_2$ is $2$-dimensional (and is contained in the facet $\mathcal{F}_{123}$), the center manifold at $E_2$ is $1$-dimensional (and is contained in the edge $\mathcal{F}_{24}$), and since the flow on the edge $\mathcal{F}_{24}$ goes from $E_2$ to $E_4$, $E_2$ is repelling on the center manifold. By the reduction principle (see~\cite[Theorem~5.2]{kuznetsov:2004}), $E_2$ is topologically a saddle in $\Delta_4$, and hence, cannot attract any orbit from the interior of $\Delta_4$. Therefore, we conclude that the statement of the lemma holds for the corner $E_2$.

\begin{figure}
\begin{center}
\begin{tabular}{cc}
\begin{tikzpicture}[radius=0.05]
\coordinate (P1) at (2  ,-2   );
\coordinate (P2) at (1  ,-1   );
\coordinate (P4) at (1,-1.5 );
\draw (0,0) -- (3,-3);
\draw (1,0) -- (1,-3);
\draw [fill] (P1) circle; \node [above right] at (P1) {$P_1$};
\draw [fill] (P2) circle; \node [above right] at (P2) {$P_2$};
\draw [fill] (P4) circle; \node [left]       at (P4) {$P_4$};
\end{tikzpicture}
&
\begin{tikzpicture}[scale=1]
\coordinate (E1)  at (0   ,0   );
\coordinate (E2)  at (3   ,0   );
\coordinate (E3)  at (1.5 ,2.4 );
\coordinate (E4)  at (1.5 ,0.72);
\coordinate (arr42) at (1.875,0.54);
\coordinate (arr24a) at (2.625+.036,0.18+.075);
\coordinate (arr24b) at (2.625-.036,0.18-.075);
\coordinate (arr21) at (2.5,0);
\coordinate (arr23) at (2.75,0.4);
\draw [semithick] (E1) -- (E4);
\draw [semithick] (E2) -- (E3);
\draw [semithick] (E1) -- (E2);
\draw [semithick] (E2) -- (E4);
\draw [semithick] (E4) -- (E3);
\draw [semithick] (E3) -- (E1);
\node [below left]  at (E1) {\small $E_1$};
\node [below right] at (E2) {\small $E_2$};
\node [above]       at (E3) {\small $E_3$};
\node [above left]  at (E4) {\small $E_4$};
\draw [fill,blue] (E2) circle [radius=0.05];
\draw [blue,draw opacity=0,semithick,-stealth] (E2) -- (arr42);
\draw [blue,draw opacity=0,semithick,-stealth] (E1) -- (arr21);
\draw [blue,draw opacity=0,semithick,-stealth] (E3) -- (arr23);
\draw [blue,semithick,-] (arr24a) -- (arr24b);
\end{tikzpicture}
\\
\end{tabular}
\end{center}
\caption{Illustration of the proof that $a_4 = a_2$ and $c_3 < 0$ imply $b_4 < b_2$ (left panel) and no orbit from the interior of $\Delta_4$ can converge to $E_2$ (right panel).}
\label{fig:no_approach_E2}
\end{figure}
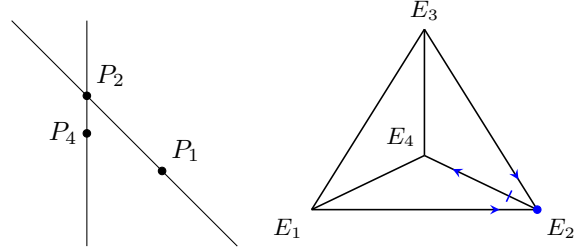

By the application of the rotation $\mathbf{r}_2$, we immediately obtain that the statement of the lemma holds also for $E_1$.

Since $b_3 < b_4$, the eigenvalue at $E_3$ (respectively, $E_4$) in the direction $E_4$ (respectively, $E_3$) is positive. Thus, by Lemma \ref{lem:saturated}, none of $E_3$ and $E_4$ can attract an orbit from $\Delta_4^\circ$, and hence, from $S$.
\end{proof}

Next, we describe the behaviour on the facets of the tetrahedron $\Delta_4$. We are especially interested in the behaviour around the edge equilibria $E_{12}$ and $E_{34}$.

\begin{lemma}\label{lem:facets}
The following four statements hold true.
\begin{enumerate}[(i)]
\item\label{it:facet234} If $c_1 \neq 0$, there is no equilibrium in the open facet $\mathcal{F}_{234}^\circ$ and there, in a neighbourhood of $E_{34}$, we have $\sgn \dot{x}_2 = \sgn(-c_1)\sgn(b_4-b_3)$.
\item\label{it:facet134} If $c_2 \neq 0$, there is no equilibrium in the open facet $\mathcal{F}_{134}^\circ$ and there, in a neighbourhood of $E_{34}$, we have $\sgn \dot{x}_1 = \sgn(-c_2)\sgn(b_4-b_3)$.
\item\label{it:facet124} If $c_3 \neq 0$, there is no equilibrium in the open facet $\mathcal{F}_{124}^\circ$ and there, in a neighbourhood of $E_{12}$, we have $\sgn \dot{x}_4 = \sgn(+c_3)\sgn(a_2-a_1)$.
\item\label{it:facet123} If $c_4 \neq 0$, there is no equilibrium in the open facet $\mathcal{F}_{123}^\circ$ and there, in a neighbourhood of $E_{12}$, we have $\sgn \dot{x}_3 = \sgn(+c_4)\sgn(a_2-a_1)$.
\end{enumerate}
\end{lemma}
\begin{proof}
First, we prove \eqref{it:facet234}. The dynamics on the facet $\mathcal{F}_{234}$ is given by the replicator dynamics with matrix
\begin{align*}
B = \begin{pmatrix}
0       & b_2-b_3 & b_4-b_2 \\
a_2-a_3 & 0       & b_4-b_3 \\
a_2-a_4 & b_4-b_3 & 0
\end{pmatrix}
\end{align*}
and variable $\widetilde{x}=(x_2,x_3,x_4)$. We claim that the function $V:\mathcal{F}_{234}^\circ\to\mathbb{R}$, defined by
\begin{align*}
V(x_2,x_3,x_4) = (b_4-b_3)\log x_2 + (b_2-b_4)\log x_3 + (b_3-b_2)\log x_4,
\end{align*}
is a Lyapunov function. Indeed,
\begin{align*}
(V(\widetilde{x}))^{\cdot} = (b_4-b_3)(B\widetilde{x})_2 + (b_2-b_4)(B\widetilde{x})_3 + (b_3-b_2)(B\widetilde{x})_4 = x_2(-c_1).
\end{align*}
Thus, there is no equilibrium in $\mathcal{F}_{234}^\circ$, and in a neighbourhood of $E_{34}$ in $\mathcal{F}_{234}^\circ$, we have $\sgn \dot{x}_2 = \sgn(-c_1)\sgn(b_4-b_3)$.

One can prove \eqref{it:facet134}, \eqref{it:facet124}, and \eqref{it:facet123} in a similar way. However, it is more elegant to say that \eqref{it:facet134}, \eqref{it:facet124}, and \eqref{it:facet123} follow from \eqref{it:facet234} by the application of the rotations $\mathbf{r}_2$, $\mathbf{r}_3$, and $\mathbf{r}_1$, respectively.
\end{proof}

Using (among other things) the previous lemma, we now show that none of the edge equilibria can attract an orbit from $S$.

\begin{lemma}\label{lem:edge_eq_doesnt_attract}
Assume that $\sgn J = \begin{pmatrix}
+ & - \\ + & -
\end{pmatrix}$, $\det J > 0$, and $a_4 \leq a_2 < a_1 \leq a_3$. Consider any of the $9$ cases in Figure~\ref{fig:9_cs}. Then, for each $i,j \in \{1,2,3,4\}$ with $i\neq j$, if $\mathcal{F}_{ij} \subseteq \partial S$ and $E_{ij}$ exists, it cannot attract an orbit from $S$.
\end{lemma}
\begin{proof}
By Figure~\ref{fig:9_cs}, if $\mathcal{F}_{13} \subseteq \partial S$ then $c_2 > 0$ and $c_4 < 0$. Thus, if $E_{13}$ exists, both of the external eigenvalues $\Gamma_{13}^2$ and $\Gamma_{13}^4$ are positive, see the equations \eqref{eq:Gamma_special}. Thus, by Lemma \ref{lem:saturated}, $E_{13}$ cannot attract an orbit from $\Delta_4^\circ$, and hence, from $S$.

By Figure~\ref{fig:9_cs}, if $\mathcal{F}_{24} \subseteq \partial S$ then $c_1 > 0$ and $c_3 < 0$. Thus, if $E_{24}$ exists, both of the external eigenvalues $\Gamma_{24}^1$ and $\Gamma_{24}^3$ are positive, see the equations \eqref{eq:Gamma_special}. Thus, by Lemma \ref{lem:saturated}, $E_{24}$ cannot attract an orbit from $\Delta_4^\circ$, and hence, from $S$. (Alternatively, one could prove the statement for $E_{24}$ by applying the rotation $\mathbf{r}_2$ to the statement for $E_{13}$.)

By Figure~\ref{fig:9_cs}, if $\mathcal{F}_{14} \subseteq \partial S$ then $c_2 > 0$. Since $c_2 = \Delta(134)$ (see \eqref{eq:c_Delta}), we obtain $b_1 < b_4$ (see the left panel in Figure~\ref{fig:no_eq_F14_F_23}). Since $a_4 < a_1$ also holds, the flow on $\mathcal{F}_{14}$ goes from $E_4$ to $E_1$ and there is no equilibrium in $\mathcal{F}_{14}^\circ$.

By Figure~\ref{fig:9_cs}, if $\mathcal{F}_{23} \subseteq \partial S$ then $c_1 > 0$. Since $c_1 = \Delta(243)$ (see \eqref{eq:c_Delta}), we obtain $b_3 < b_2$ (see the right panel in Figure~\ref{fig:no_eq_F14_F_23}). Since $a_2 < a_3$ also holds, the flow on $\mathcal{F}_{23}$ goes from $E_3$ to $E_2$ and there is no equilibrium in $\mathcal{F}_{23}^\circ$. (Alternatively, one could prove the non-existence of $E_{23}$ by applying the rotation $\mathbf{r}_2$ to the non-existence of $E_{14}$.)

\begin{figure}
\begin{center}
\begin{tabular}{cc}
\begin{tikzpicture}[radius=0.05]
\coordinate (P1) at (1.5,-2  );
\coordinate (P3) at (2.5,-2.5);
\coordinate (P4) at (0.5,-0.5);
\draw (0,0) -- (3,-3);
\draw (0.5,0) -- (0.5,-3);
\draw (2.5,0) -- (2.5,-3);
\draw [fill] (P1) circle; \node [below left]  at (P1) {$P_1$};
\draw [fill] (P3) circle; \node [above right] at (P3) {$P_3$};
\draw [fill] (P4) circle; \node [above right] at (P4) {$P_4$};
\end{tikzpicture}
&
\begin{tikzpicture}[radius=0.05]
\coordinate (P2) at (1.5,-1  );
\coordinate (P3) at (2.5,-2.5);
\coordinate (P4) at (0.5,-0.5);
\draw (0,0) -- (3,-3);
\draw (0.5,0) -- (0.5,-3);
\draw (2.5,0) -- (2.5,-3);
\draw [fill] (P2) circle; \node [above right] at (P2) {$P_2$};
\draw [fill] (P3) circle; \node [below left]  at (P3) {$P_3$};
\draw [fill] (P4) circle; \node [below left]  at (P4) {$P_4$};
\end{tikzpicture}
\\
\end{tabular}
\end{center}
\caption{Illustration of the proof of the facts that $c_2 > 0$ implies $b_1 < b_4$ (left panel) and $c_1 > 0$ implies $b_3 < b_2$ (right panel).}
\label{fig:no_eq_F14_F_23}
\end{figure}
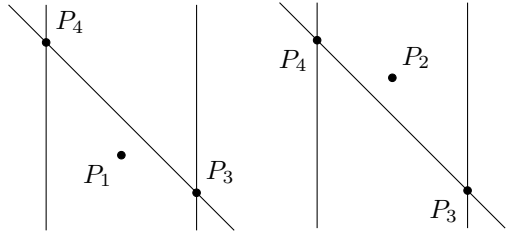

The edge $\mathcal{F}_{12}$ is not part of $\partial S$ in any of the $9$ cases in Figure~\ref{fig:9_cs}.

Assume now that $\mathcal{F}_{34} \subseteq \partial S$. Since $b_4-b_3 > 0$, the edge equilibrium $E_{34}$ is stable within the edge $\mathcal{F}_{34}$, with eigenvalue $-\frac{b_4-b_3}{2} < 0$, see Appendix \ref{sec:app_replicator}. Both of the external eigenvalues $\Gamma^1_{34}$ and $\Gamma^2_{34}$ are zero. This follows from the general formula in Appendix \ref{sec:app_special}, but also from the existence of the line of equilibria $\left(\frac{\eps}{2}, \frac{\eps}{2},  \frac{1-\eps}{2}, \frac{1-\eps}{2}\right)$ for $0 \leq \eps \leq 1$. Therefore, the stable manifold at $E_{34}$ is $1$-dimensional (and is contained in the edge $\mathcal{F}_{34}$) and there is a (not necessarily unique) $2$-dimensional center manifold through $E_{34}$, transversal to the edge $\mathcal{F}_{34}$, and containing the equilibria $\left(\frac{\eps}{2}, \frac{\eps}{2}, \frac{1-\eps}{2}, \frac{1-\eps}{2}\right)$ for small $\eps>0$.

Now assume further that $\sgn(c_1,c_2) = (+,-)$. The invariant surfaces $\{Q=d\}$ can be expressed as $x_2 = \left(d x_1^{-c_1}x_3^{-c_3}x_4^{-c_4}\right)^{\frac{1}{c_2}}$, which is approximately $d'x_1^{-\frac{c_1}{c_2}}$ near $E_{34}$. Since, by the inequality \eqref{eq:c1_c2_detJ}, we have $-\frac{c_1}{c_2}>1$, the invariant surfaces $\{Q=d\}$ are tangent to the facet $\mathcal{F}_{134}$ at $E_{34}$. On the facet $\mathcal{F}_{134}^\circ$ near $E_{34}$, by Lemma \ref{lem:facets} \eqref{it:facet134}, we have $\sgn \dot{x}_1 = \sgn(-c_2)\sgn(b_4-b_3)$, which is positive. The invariant surfaces $\{Q=d\}$ intersect the $2$-dimensional center manifold in a family of curves tangent to the facet $\mathcal{F}_{134}$. Then, by continuity, $\dot{x}_1>0$ holds on these curves near $E_{34}$ (because the only equilibria in $\Delta_4^\circ$ are on the line segment from $E_{12}$ to $E_{34}$), in particular, $\dot{x}_1>0$ on the intersection $\sigma$ of $S$ with the center manifold. Hence, the flow on $\sigma$ moves away from $E_{34}$ (see Figure~\ref{fig:center_manif_E34}). Then, by the reduction principle (see~\cite[Theorem~5.2]{kuznetsov:2004}), applied to the $2$-dimensional flow on $\overline{S}$, $E_{34}$ is topologically a saddle in $\overline{S}$, and hence, does not attract any orbit from $S$. To arrive at the same conclusion in case $\sgn(c_1,c_2) = (-,+)$, one can apply the rotation $\mathbf{r}_2$.

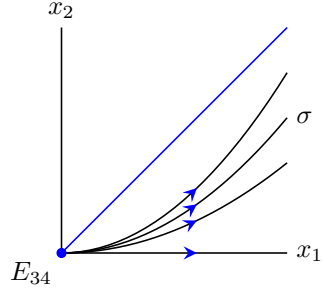
\begin{figure}
\begin{center}
\begin{tikzpicture}[scale=3]
\coordinate (00) at (0,0);
\coordinate (10) at (1,0);
\coordinate (01) at (0,1);
\coordinate (11) at (1,1);
\coordinate (sigma) at (1,0.6);
\coordinate (arraxis) at (0.6,0);
\coordinate (arr1tail) at (0.6-.1,0.4*0.6^2-.2*0.4*0.6);
\coordinate (arr1head) at (0.6   ,0.4*0.6^2           );
\coordinate (arr2tail) at (0.6-.1,0.6*0.6^2-.2*0.6*0.6);
\coordinate (arr2head) at (0.6   ,0.6*0.6^2           );
\coordinate (arr3tail) at (0.6-.1,0.8*0.6^2-.2*0.8*0.6);
\coordinate (arr3head) at (0.6   ,0.8*0.6^2           );
\draw [blue,semithick] (00) -- (11);
\draw [semithick] (00) -- (10);
\draw [semithick] (00) -- (01);
\draw [semithick,domain=0:1] plot (\x,0.4*\x^2);
\draw [semithick,domain=0:1] plot (\x,0.6*\x^2);
\draw [semithick,domain=0:1] plot (\x,0.8*\x^2);
\node [below left] at (00) {$E_{34}$};
\node [right] at (10) {$x_1$};
\node [above] at (01) {$x_2$};
\node [right] at (sigma) {$\sigma$};
\draw [blue,fill,radius=0.02] (00) circle;
\draw [blue,draw opacity=0,very thick,-stealth] (00) -- (arraxis);
\draw [blue,draw opacity=0,very thick,-stealth] (arr1tail) -- (arr1head);
\draw [blue,draw opacity=0,very thick,-stealth] (arr2tail) -- (arr2head);
\draw [blue,draw opacity=0,very thick,-stealth] (arr3tail) -- (arr3head);
\end{tikzpicture}
\end{center}
\caption{The flow on the center manifold at $E_{34}$ near $\mathcal{F}_{134}$.}
\label{fig:center_manif_E34}
\end{figure}
\end{proof}

The following theorem is a special case of \cite[Theorem 3.1]{butler:waltman:1986}.

\begin{theorem} \label{thm:butler}
Consider the ODE~\eqref{eq:ode_A} restricted to $\overline{S}$ with $\det J > 0$. If
\begin{enumerate}[(i)]
\item there are only finitely many equilibria in $\partial S$,
\item no equilibrium in $\partial S$ attracts an orbit from $S$, and
\item $\partial S$ does not form a heteroclinic cycle (between the equilibria in $\partial{S}$)
\end{enumerate}
then the ODE~\eqref{eq:ode_A} restricted to $\overline S$ is permanent.
\end{theorem}

The following lemma is obvious.

\begin{lemma}\label{lem:at_most_one_edge_eq}
Fix $i,j \in \{1,2,3,4\}$ with $i \neq j$. Then the following two statements hold.
\begin{enumerate}[(i)]
\item If $P_i \neq P_j$ then there is at most one equilibrium in the edge $\mathcal{F}_{ij}^\circ$.
\item If $P_i = P_j$ then $c_k = 0$ for all $k \in \{1,2,3,4\} \setminus \{i,j\}$.
\end{enumerate}
\end{lemma}

The next lemma covers the case $\sgn c = (+,+,-,-)$, i.e., the situation in the $1$st row in Figure~\ref{fig:9_cs}, the quadrangle case. It is permanent.

\begin{lemma}\label{lem:quadrangle_not_heteroclinic}
Assume $\sgn J = \begin{pmatrix}
+ & - \\ + & -
\end{pmatrix}$, $\det J > 0$, and that $a_4 \leq a_2 < a_1 \leq a_3$ and $\sgn c = (+,+,-,-)$. Further, assume that at least one of the two inequalities $b_1 \leq b_3$ and $b_4 \leq b_2$ is violated. Then the ODE~\eqref{eq:ode_exp} is permanent.
\end{lemma}
\begin{proof}
We apply Theorem \ref{thm:butler}. By Lemma \ref{lem:at_most_one_edge_eq}, there are only finitely many equilibria on $\partial S$. By Lemmata \ref{lem:corners_doesnt_attract} and \ref{lem:edge_eq_doesnt_attract}, no boundary equilibrium attracts an orbit from $S$. As can be read from Figure~\ref{fig:quadrangle_not_heteroclinic}, neither in case $b_3 < b_1$ nor in case $b_2 < b_4$ the boundary of $\partial S$ forms a heteroclinic cycle.

\begin{figure}
\begin{center}
\begin{tabular}{cc}
\begin{tikzpicture}
\coordinate (E1)  at (0   ,0   );
\coordinate (E2)  at (3   ,0   );
\coordinate (E3)  at (1.5 ,2.4 );
\coordinate (E4)  at (1.5 ,0.72 );
\coordinate (arr14) at (0.375,0.18);
\coordinate (arr23) at (2.625,0.6);
\coordinate (arr31) at (1.125,1.8);
\coordinate (caption) at (0,2.4);
\draw [semithick] (E1) -- (E2);
\draw [semithick] (E3) -- (E4);
\draw [red, semithick] (E1) -- (E4);
\draw [red, semithick] (E4) -- (E2);
\draw [red, semithick] (E2) -- (E3);
\draw [red, semithick] (E3) -- (E1);
\node [below left]  at (E1) {\small $E_1$};
\node [below right] at (E2) {\small $E_2$};
\node [above]       at (E3) {\small $E_3$};
\node [above left]  at (E4) {\small $E_4$};
\draw [blue,draw opacity=0,semithick,-stealth] (E4) -- (arr14);
\draw [blue,draw opacity=0,semithick,-stealth] (E3) -- (arr23);
\draw [blue,draw opacity=0,semithick,-stealth] (E3) -- (arr31);
\node at (caption) {\small if $b_3 < b_1$};
\end{tikzpicture}
&
\begin{tikzpicture}
\coordinate (E1)  at (0   ,0   );
\coordinate (E2)  at (3   ,0   );
\coordinate (E3)  at (1.5 ,2.4 );
\coordinate (E4)  at (1.5 ,0.72 );
\coordinate (arr14) at (0.375,0.18);
\coordinate (arr23) at (2.625,0.6);
\coordinate (arr42) at (1.875,0.54);
\coordinate (caption) at (0,2.4);
\draw [semithick] (E1) -- (E2);
\draw [semithick] (E3) -- (E4);
\draw [red, semithick] (E1) -- (E4);
\draw [red, semithick] (E4) -- (E2);
\draw [red, semithick] (E2) -- (E3);
\draw [red, semithick] (E3) -- (E1);
\node [below left]  at (E1) {\small $E_1$};
\node [below right] at (E2) {\small $E_2$};
\node [above]       at (E3) {\small $E_3$};
\node [above left]  at (E4) {\small $E_4$};
\draw [blue,draw opacity=0,semithick,-stealth] (E4) -- (arr14);
\draw [blue,draw opacity=0,semithick,-stealth] (E3) -- (arr23);
\draw [blue,draw opacity=0,semithick,-stealth] (E4) -- (arr42);
\node at (caption) {\small if $b_2 < b_4$};
\end{tikzpicture}
\\
\end{tabular}
\end{center}
\caption{Under the assumptions of Lemma \ref{lem:quadrangle_not_heteroclinic}, $\partial S$ does not form a heteroclinic cycle.}
\label{fig:quadrangle_not_heteroclinic}
\end{figure}
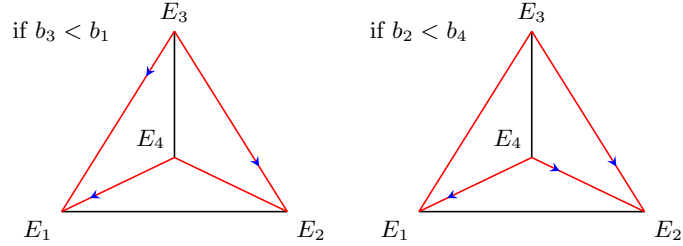
\end{proof}

The next lemma covers the cases, where $\sgn c$ is one of $(+,-,-,-)$, $(+,0,-,-)$, $(-,+,-,-)$, $(0,+,-,-)$, i.e., the situations in the $2$nd and $3$rd rows in Figure~\ref{fig:9_cs}. They are all permanent. Lemmata \ref{lem:quadrangle_not_heteroclinic} and \ref{lem:triangle} together conclude the proof of the fact that $\sgn(c_3,c_4) = (-,-)$ implies permanence, it is the case (C1a) in Theorem \ref{thm:oppositesign}.

\begin{lemma}\label{lem:triangle}
Assume $\sgn J = \begin{pmatrix}
+ & - \\ + & -
\end{pmatrix}$, $\det J > 0$, and that $a_4 \leq a_2 < a_1 \leq a_3$. Further, assume that $\sgn(c_3,c_4) = (-,-)$ and $\sgn(c_1,c_2)\neq(+,+)$. Then the ODE~\eqref{eq:ode_exp} is permanent.
\end{lemma}
\begin{proof}
By the inequality \eqref{eq:c1_c2_detJ} and the assumption $\sgn(c_1,c_2)\neq(+,+)$, $\sgn(c_1,c_2)$ is one of $(+,-)$, $(+,0)$, $(-,+)$, $(0,+)$.

To prove permanence in case $\sgn c = (+,-,-,-)$, we apply Theorem \ref{thm:butler}. By Lemma \ref{lem:at_most_one_edge_eq}, there are only finitely many equilibria on $\partial S$. By Lemmata \ref{lem:corners_doesnt_attract} and \ref{lem:edge_eq_doesnt_attract}, no boundary equilibrium attracts an orbit from $S$. Since there exists an edge equilibrium on $\partial S$, namely $E_{34}$, $\partial S$ does not form a heteroclinic cycle.

To prove permanence in case $\sgn c = (-,+,-,-)$, one can argue similarly as in the above paragraph. Alternatively, one may apply the rotation $\mathbf{r}_2$ to the previous case.

Next, we prove permanence in case $\sgn c = (+,0,-,-)$. Then $\partial S$ consists of the two edges $\mathcal{F}_{42}$, $\mathcal{F}_{23}$ and the curve
\begin{align*}
\mathcal{C}_{34}^1 = \{x \in \Delta_4~|~ x_2 = 0 \text{ and } x_1^{|c_1|} = x_3^{|c_3|}x_4^{|c_4|}\} \subseteq \mathcal{F}_{134},
\end{align*}
which connects $E_3$ and $E_4$.
Note that $a_1 = a_3$ is not possible, because $c_2 = 0$ and $c_4 \neq 0$. Since $a_4 < a_1 < a_3$ together with $b_3 < b_4$ and $c_2 = 0$ (i.e., the three points $P_1$, $P_3$, $P_4$ lie on a line) imply that $b_3 < b_1 < b_4$, see the left panel in Figure~\ref{fig:triangle_134}. Thus, the sign of the matrix corresponding to the facet $\mathcal{F}_{134}$ is
\begin{align*}
\begin{pmatrix}
0 & + & + \\
+ & 0 & + \\
- & + & 0 
\end{pmatrix}.
\end{align*}
Hence, $E_3$ is repelling both in the direction of $E_1$ and of $E_4$, and, similarly, $E_4$ is repelling both in the direction of $E_1$ and of $E_3$. Thus, the flow on $\mathcal{C}_{34}^1$ goes away both from $E_3$ and from $E_4$. Furthermore, there is no equilibrium on $\mathcal{F}_{14}$, the flow goes from $E_4$ to $E_1$, and there exists an equilibrium $E_{13}$ on the edge $\mathcal{F}_{13}$. Using $c_2 = 0$, a short calculation shows that there is a line of equilibria connecting $E_{13}$ and $E_{34}$, and this line intersects $\mathcal{C}_{34}^1$ at a unique equilibrium $E_{34}^1$. (All equilibria in $\mathcal{F}_{134}^\circ$ lie on this line.) See the right panel in Figure~\ref{fig:triangle_134} for the dynamics on the facet $\mathcal{F}_{134}$. The external eigenvalue at $E_{13}$ in the direction $E_2$ is positive (because $\sgn \Gamma_{13}^2  = - \sgn c_4  = +1$), while the external eigenvalue at $E_{34}$ in the direction $E_2$ is zero, see Appendix \ref{sec:app_special}. It is a general fact that the eigenvalue in the direction of $E_2$ changes linearly from $\Gamma_{34}^2$ to $\Gamma_{13}^2$ while travelling on the line of equilibria from $E_{34}$ to $E_{13}$. Thus, the external eigenvalue at $E_{34}^1$ is positive, and hence, $E_{34}^1$ is not saturated. Therefore, by Lemma \ref{lem:saturated}, it cannot attract an orbit from $\Delta_4^\circ$, and hence, from $S$. Permanence in case $\sgn c = (+,0,-,-)$ then follows immediately from Theorem \ref{thm:butler}.

To prove permanence in case $\sgn c = (0,+,-,-)$, one can argue similarly as in the above paragraph. Alternatively, one may apply the rotation $\mathbf{r}_2$ to the previous case.

\begin{figure}
\begin{center}
\begin{tabular}{cc}
\begin{tikzpicture}[radius=0.05]
\coordinate (P1) at (1.5,-1.5);
\coordinate (P3) at (2.5,-2.5);
\coordinate (P4) at (0.5,-0.5);
\draw (0,0) -- (3,-3);
\draw [fill] (P1) circle; \node [above right] at (P1) {$P_1$};
\draw [fill] (P3) circle; \node [above right] at (P3) {$P_3$};
\draw [fill] (P4) circle; \node [above right] at (P4) {$P_4$};
\end{tikzpicture}
&
\begin{tikzpicture}[scale=1,radius=0.05]
\coordinate (E1)  at (0,0);
\coordinate (E3)  at (3,3);
\coordinate (E4)  at (3,0);
\coordinate (E34)  at (3,1.5);
\coordinate (E13)  at (1.25,1.25);
\coordinate (arr34) at (3,2.25);
\coordinate (arr43) at (3,0.75);
\coordinate (arr14) at (0.75,0);
\coordinate (arr41) at (2.25,0);
\coordinate (arr13) at (0.75,0.75);
\coordinate (arr31) at (2.25,2.25);
\coordinate (arr3curvetail) at (-0.9239*2.1213+4.5+0.4081,0.3827*2.1213+1.5+0.9129);
\coordinate (arr3curvehead) at (-0.9239*2.1213+4.5       ,0.3827*2.1213+1.5       );
\coordinate (arr4curvetail) at (-0.9239*2.1213+4.5+0.4081,-0.3827*2.1213+1.5-0.9129);
\coordinate (arr4curvehead) at (-0.9239*2.1213+4.5       ,-0.3827*2.1213+1.5       );
\coordinate (E341) at (2.3805,2.3805/7+15/14);
\draw [semithick] (E1) -- (E4);
\draw [semithick] (E4) -- (E3);
\draw [semithick] (E3) -- (E1);
\draw [blue,semithick] (E34) -- (E13);
\draw [red, semithick] (E3) to [out=-135,in=135] (E4);
\node [below left]  at (E1) {\small $E_1$};
\node [above right] at (E3) {\small $E_3$};
\node [below right] at (E4) {\small $E_4$};
\node [right]       at (E34) {\small $E_{34}$};
\node [above left]  at (E13) {\small $E_{13}$};
\node [below left]  at (E341) {\small $E_{34}^1$};
\draw [fill,blue] (E1) circle;
\draw [fill,blue] (E3) circle;
\draw [fill,blue] (E4) circle;
\draw [fill,blue] (E34) circle;
\draw [fill,blue] (E13) circle;
\draw [fill,blue] (E341) circle;
\draw [blue,draw opacity=0,semithick,-stealth] (E3) -- (arr34);
\draw [blue,draw opacity=0,semithick,-stealth] (E4) -- (arr43);
\draw [blue,draw opacity=0,semithick,-stealth] (E4) -- (arr14);
\draw [blue,draw opacity=0,semithick,-stealth] (E4) -- (arr41);
\draw [blue,draw opacity=0,semithick,-stealth] (E1) -- (arr13);
\draw [blue,draw opacity=0,semithick,-stealth] (E3) -- (arr31);
\draw [blue,draw opacity=0,semithick,-stealth] (arr3curvetail) -- (arr3curvehead);
\draw [blue,draw opacity=0,semithick,-stealth] (arr4curvetail) -- (arr4curvehead);
\end{tikzpicture}
\\
\end{tabular}
\end{center}
\caption{The relative position of the three points $P_1$, $P_3$, $P_4$ and the behaviour on the facet $\mathcal{F}_{134}$ when $\sgn c = (+,0,-,-)$ and $a_4 \leq a_2 < a_1 \leq a_3$.}
\label{fig:triangle_134}
\end{figure}
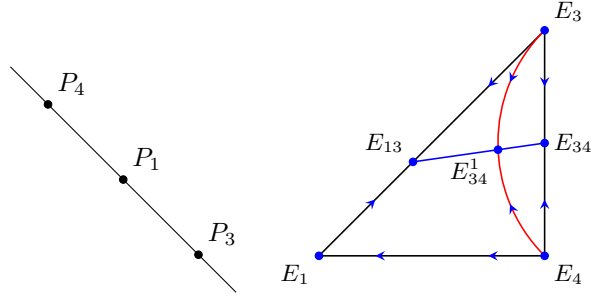
\end{proof}

The next lemma covers the cases, where $\sgn c$ is one of $(+,-,-,0)$ and $(-,+,0,-)$, i.e., the situations in the left panels in the $4$th and $5$th rows in Figure~\ref{fig:9_cs}. It is the case (C1b) in Theorem \ref{thm:oppositesign}.

\begin{lemma} \label{lem:c3_or_c4_zero}
Assume $\sgn J = \begin{pmatrix}
+ & - \\ + & -
\end{pmatrix}$, $\det J > 0$, and that $a_4 \leq a_2 < a_1 \leq a_3$ and either $\sgn c = (+,-,-,0)$ or $\sgn c = (-,+,0,-)$. Then the ODE~\eqref{eq:ode_exp} is permanent if and only if
\begin{align} \label{eq:c3_or_c4_zero_permanence_condition}
-\frac{a_1-a_2}{b_1-b_2} < \frac{(L+1)^{L+1}}{L^L},
\text{ where } L = \begin{cases} \frac{c_2}{c_3}, & \text{ if } \sgn c = (+,-,-,0), \\ \frac{c_1}{c_4}, & \text{ if } \sgn c = (-,+,0,-).\end{cases}
\end{align}
\end{lemma}
\begin{proof}
Assume first that $\sgn c = (+,-,-,0)$. Then $\partial S$ consists of the two edges $\mathcal{F}_{34}$, $\mathcal{F}_{42}$ and the curve
\begin{align*}
\mathcal{C}_{23}^1 = \{x \in \Delta_4~|~ x_4 = 0 \text{ and } x_1^{|c_1|} = x_2^{|c_2|}x_3^{|c_3|}\} \subseteq \mathcal{F}_{123},
\end{align*}
which connects $E_2$ and $E_3$. Note that $a_1 = a_3$ is not possible, because $c_4 = 0$ and $c_2 \neq 0$. Since $a_2 < a_1 < a_3$ together with $b_1 < b_2$ and $c_4 = 0$ (i.e., the three points $P_1$, $P_2$, $P_3$ lie on a line) imply that $b_3 < b_1 < b_2$, see the upper left panel in Figure~\ref{fig:triangle_123}. Thus, the sign of the matrix corresponding to the triangle $\mathcal{F}_{123}$ is
\begin{align*}
\begin{pmatrix}
0 & - & + \\
- & 0 & + \\
+ & - & 0 
\end{pmatrix}.
\end{align*}
Hence, $E_1$ is a saddle, $E_2$ is an attractor, and $E_3$ is a repeller. Thus, the flow on $\mathcal{C}_{23}^1$ goes away from $E_3$ and goes towards $E_2$. Furthermore, there is no equilibrium on $\mathcal{F}_{23}$, the flow goes from $E_3$ to $E_2$, and there exist equilibria $E_{13}$ and $E_{12}$ on the edges $\mathcal{F}_{13}$ and $\mathcal{F}_{12}$, respectively. Using $c_4 = 0$, a short calculation shows that there is a line of equilibria $\mathcal{E}$ connecting $E_{12}$ and $E_{13}$. (All equilibria in $\mathcal{F}_{123}^\circ$ lie on $\mathcal{E}$.) The line $\mathcal{E}$ intersects $\mathcal{C}_{23}^1$ at either $0$, $1$, or $2$ points. See Figure~\ref{fig:triangle_123} for the dynamics on the facet $\mathcal{F}_{123}$ in these three cases.

If the intersection of $\mathcal{E}$ and $\mathcal{C}_{23}^1$ is empty then, by Theorem \ref{thm:butler}, permanence of the ODE~\eqref{eq:ode_A} follows. On the other hand, if $\mathcal{E}$ and $\mathcal{C}_{23}^1$ intersect each other, say at $E_{23}^1$, then the system is not permanent. Indeed, in this case the external eigenvalue at $E_{23}^1$ in the direction $E_4$ is negative, since it is a convex combination of $\Gamma_{12}^4=0$ and $\Gamma_{13}^4$, the latter being negative (because, by \eqref{eq:Gamma_special}, $\sgn \Gamma_{13}^4 = +\sgn c_2 = -1$). Applying the Stable Manifold Theorem to the flow restricted to $\overline{S}$, there exists at least one orbit in $S$ that converges to $E_{23}^1$. Thus, the ODE~\eqref{eq:ode_exp} is permanent if only if $\mathcal{C}_{23}^1$ does not intersect $\mathcal{E}$. The latter is  equivalent to
\begin{align*}
x_1^{c_1}x_2^{c_2}x_3^{c_3} \not= 1 \text{ for all } x \in \mathcal{E},
\end{align*}
or 
\begin{align} \label{eq:empty_intersection}
c_1 \log x_1 + c_2 \log x_2 + c_3 \log x_3  \not= 0 \text{ for all } x \in \mathcal{E}.
\end{align}
With $\lambda = \frac{b_1-b_3}{a_3-a_1 + b_1-b_3}$, we have $E_{13} = (\lambda,0,1-\lambda, 0)$. Since $E_{12} = \left(\frac{1}{2},\frac{1}{2},0,0\right)$, the statement \eqref{eq:empty_intersection} is equivalent to $g(\eps) \neq 0$ for all $0<\eps<1$, where
\begin{align} \label{eq:eps_lambda}
g(\eps)= c_1  \log \left(\eps\lambda+\frac{1-\eps}{2}\right) + c_2 \log \left(\frac{1-\eps}{2}\right) + c_3 \log \left(\eps(1-\lambda)\right).
\end{align}
Since $g(\eps) \to + \infty$ as $\eps \to 0$ or $1$, this is equivalent to $g(\eps) > 0$ for all $ 0 < \eps < 1$. A short calculation shows that $g$ attains its minimum at $\eps^* = \frac{c_3}{c_3 + 2\lambda c_2}$. Note that
\begin{align*}
g(\eps^*) = c_3 \log \frac{1-\lambda}{\lambda} + \log \frac{(-c_2)^{c_2}(-c_3)^{c_3}}{(-c_2-c_3)^{c_2+c_3}}.
\end{align*}
Since $-\frac{a_1-a_2}{b_1-b_2}=-\frac{a_1-a_3}{b_1-b_3}=\frac{1-\lambda}{\lambda}$, the fact 
$g(\eps^*)$ is positive is equivalent to the upper case in \eqref{eq:c3_or_c4_zero_permanence_condition}.

To prove the statement in case $\sgn c = (-,+,0,-)$, one can argue similarly. Alternatively, one may apply the rotation $\mathbf{r}_2$ to the case $\sgn c = (+,-,-,0)$.

\begin{figure}
\begin{center}
\begin{tabular}{cc}
\begin{tikzpicture}[radius=0.05]
\coordinate (P1) at (1.5,-1.5);
\coordinate (P2) at (0.5,-0.5);
\coordinate (P3) at (2.5,-2.5);
\draw (0,0) -- (3,-3);
\draw [fill] (P1) circle; \node [above right] at (P1) {$P_1$};
\draw [fill] (P2) circle; \node [above right] at (P2) {$P_2$};
\draw [fill] (P3) circle; \node [above right] at (P3) {$P_3$};
\end{tikzpicture}
&
\begin{tikzpicture}[scale=1,radius=0.05]
\coordinate (E1)  at (0,0);
\coordinate (E2)  at (3,0);
\coordinate (E3)  at (1.5,2.5981);
\coordinate (E12)  at (0.5*3,0);
\coordinate (E13)  at (0.5*1.5,0.5*2.5981);
\coordinate (arr32) at (0.25*3+0.75*1.5,0.75*2.5981);
\coordinate (arr23) at (0.75*3+0.25*1.5,0.25*2.5981);
\coordinate (arr13) at (0.125*3,0.25*2.5981);
\coordinate (arr12) at (0.25*3,0);
\coordinate (arr21) at (0.75*3,0);
\coordinate (arr31) at (0.375*3,0.75*2.5981);
\coordinate (arr3curvetail) at (1.315+.2165,1.9486+1.7185);
\coordinate (arr3curvehead) at (1.315,1.9486);
\coordinate (arr2curvetail) at (2.3-1,0.21+0.5);
\coordinate (arr2curvehead) at (2.3  ,0.21);
\draw [semithick] (E1) -- (E2);
\draw [semithick] (E2) -- (E3);
\draw [semithick] (E3) -- (E1);
\draw [blue,semithick] (E12) -- (E13);
\draw [red, semithick] (E3) to [out=-120,in=180] (E2);
\node [below left]  at (E1) {\small $E_1$};
\node [below right] at (E2) {\small $E_2$};
\node [above]       at (E3) {\small $E_3$};
\node [below]       at (E12) {\small $E_{12}$};
\node [above left]  at (E13) {\small $E_{13}$};
\draw [fill,blue] (E1) circle;
\draw [fill,blue] (E2) circle;
\draw [fill,blue] (E3) circle;
\draw [fill,blue] (E12) circle;
\draw [fill,blue] (E13) circle;
\draw [blue,draw opacity=0,semithick,-stealth] (E3) -- (arr32);
\draw [blue,draw opacity=0,semithick,-stealth] (E3) -- (arr23);
\draw [blue,draw opacity=0,semithick,-stealth] (E2) -- (arr12);
\draw [blue,draw opacity=0,semithick,-stealth] (E1) -- (arr13);
\draw [blue,draw opacity=0,semithick,-stealth] (E3) -- (arr31);
\draw [blue,draw opacity=0,semithick,-stealth] (E1) -- (arr21);
\draw [blue,draw opacity=0,semithick,-stealth] (arr3curvetail) -- (arr3curvehead);
\draw [blue,draw opacity=0,semithick,-stealth] (arr2curvetail) -- (arr2curvehead);
\end{tikzpicture}
\\
\begin{tikzpicture}[scale=1,radius=0.05]
\coordinate (E1)  at (0,0);
\coordinate (E2)  at (3,0);
\coordinate (E3)  at (1.5,2.5981);
\coordinate (E12)  at (0.7*3,0);
\coordinate (E13)  at (0.7*1.5,0.7*2.5981);
\coordinate (aux) at (0.7*2.25,0.7*2.5981/2);
\coordinate (arr32) at (0.25*3+0.75*1.5,0.75*2.5981);
\coordinate (arr23) at (0.75*3+0.25*1.5,0.25*2.5981);
\coordinate (arr13) at (0.125*3,0.25*2.5981);
\coordinate (arr12) at (0.25*3,0);
\coordinate (arr21) at (0.82*3,0);
\coordinate (arr31) at (0.8*3/2,0.8*2.5981);
\coordinate (arr3curvetail) at (1.295+.2,1.9486+1.7);
\coordinate (arr3curvehead) at (1.295,1.9486);
\coordinate (arr2curvetail) at (2.3-1,0.16+0.5);
\coordinate (arr2curvehead) at (2.3  ,0.16);
\draw [semithick] (E1) -- (E2);
\draw [semithick] (E2) -- (E3);
\draw [semithick] (E3) -- (E1);
\draw [blue,semithick] (E12) -- (E13);
\draw [red, semithick] (E3) to [out=-120,in=120] (aux);
\draw [red, semithick] (aux) to [out=-60,in=180] (E2);
\node [below left]  at (E1) {\small $E_1$};
\node [below right] at (E2) {\small $E_2$};
\node [above]       at (E3) {\small $E_3$};
\node [below]       at (E12) {\small $E_{12}$};
\node [above left]  at (E13) {\small $E_{13}$};
\draw [fill,blue] (E1) circle;
\draw [fill,blue] (E2) circle;
\draw [fill,blue] (E3) circle;
\draw [fill,blue] (E12) circle;
\draw [fill,blue] (E13) circle;
\draw [fill,blue] (aux) circle;
\draw [blue,draw opacity=0,semithick,-stealth] (E3) -- (arr32);
\draw [blue,draw opacity=0,semithick,-stealth] (E3) -- (arr23);
\draw [blue,draw opacity=0,semithick,-stealth] (E2) -- (arr12);
\draw [blue,draw opacity=0,semithick,-stealth] (E1) -- (arr13);
\draw [blue,draw opacity=0,semithick,-stealth] (E3) -- (arr31);
\draw [blue,draw opacity=0,semithick,-stealth] (E1) -- (arr21);
\draw [blue,draw opacity=0,semithick,-stealth] (arr3curvetail) -- (arr3curvehead);
\draw [blue,draw opacity=0,semithick,-stealth] (arr2curvetail) -- (arr2curvehead);
\end{tikzpicture}
&
\begin{tikzpicture}[scale=1,radius=0.05]
\coordinate (E1)  at (0,0);
\coordinate (E2)  at (3,0);
\coordinate (E3)  at (1.5,2.5981);
\coordinate (E12)  at (0.7*3,0);
\coordinate (E13)  at (0.7*1.5,0.7*2.5981);
\coordinate (Enew1) at (0.2*0.7*3 + 0.8*0.7*1.5,0.8*0.7*2.5981);
\coordinate (Enew2) at (0.8*0.7*3 + 0.2*0.7*1.5,0.2*0.7*2.5981);
\coordinate (aux) at (0.66*2.25,0.66*2.5981/2);
\coordinate (auxtail) at (0.66*2.25+1.5/10,0.66*2.5981/2-2.5981/10);
\coordinate (arr32) at (0.25*3+0.75*1.5,0.75*2.5981);
\coordinate (arr23) at (0.75*3+0.25*1.5,0.25*2.5981);
\coordinate (arr13) at (0.125*3,0.25*2.5981);
\coordinate (arr12) at (0.25*3,0);
\coordinate (arr21) at (0.82*3,0);
\coordinate (arr31) at (0.8*3/2,0.8*2.5981);
\coordinate (arr3curvetail) at (1.25+.2,1.9+1.5);
\coordinate (arr3curvehead) at (1.25,1.9);
\coordinate (arr2curvetail) at (2.3-1,0.12+0.5);
\coordinate (arr2curvehead) at (2.3  ,0.12);
\draw [semithick] (E1) -- (E2);
\draw [semithick] (E2) -- (E3);
\draw [semithick] (E3) -- (E1);
\draw [blue,semithick] (E12) -- (E13);
\draw [red, semithick] (E3) to [out=-120,in=120] (aux);
\draw [red, semithick] (aux) to [out=-60,in=180] (E2);
\node [below left]  at (E1) {\small $E_1$};
\node [below right] at (E2) {\small $E_2$};
\node [above]       at (E3) {\small $E_3$};
\node [below]       at (E12) {\small $E_{12}$};
\node [above left]  at (E13) {\small $E_{13}$};
\draw [fill,blue] (E1) circle;
\draw [fill,blue] (E2) circle;
\draw [fill,blue] (E3) circle;
\draw [fill,blue] (E12) circle;
\draw [fill,blue] (E13) circle;
\draw [fill,blue] (Enew1) circle;
\draw [fill,blue] (Enew2) circle;
\draw [blue,draw opacity=0,semithick,-stealth] (E3) -- (arr32);
\draw [blue,draw opacity=0,semithick,-stealth] (E3) -- (arr23);
\draw [blue,draw opacity=0,semithick,-stealth] (E2) -- (arr12);
\draw [blue,draw opacity=0,semithick,-stealth] (E1) -- (arr13);
\draw [blue,draw opacity=0,semithick,-stealth] (E3) -- (arr31);
\draw [blue,draw opacity=0,semithick,-stealth] (E1) -- (arr21);
\draw [blue,draw opacity=0,semithick,-stealth] (auxtail) -- (aux);
\draw [blue,draw opacity=0,semithick,-stealth] (arr3curvetail) -- (arr3curvehead);
\draw [blue,draw opacity=0,semithick,-stealth] (arr2curvetail) -- (arr2curvehead);
\end{tikzpicture}
\\
\end{tabular}
\end{center}
\caption{The relative position of the three points $P_1$, $P_2$, $P_3$ and the three possible behaviours on the facet $\mathcal{F}_{123}$ when $\sgn c = (+,-,-,0)$ and $a_4 \leq a_2 < a_1 \leq a_3$.}
\label{fig:triangle_123}
\end{figure}
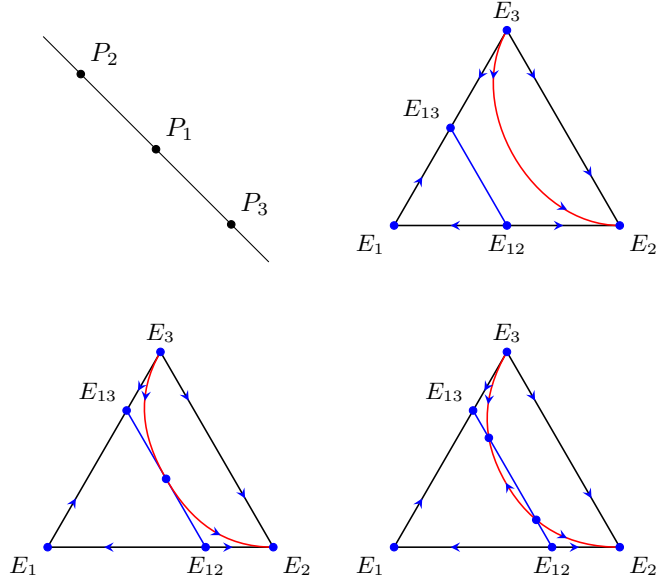
\end{proof}

Note that the function $(0,\infty) \ni L \mapsto \frac{(L+1)^{L+1}}{L^L} \in (1,\infty)$ is monotonically increasing. Since for $\sgn c = (+,-,-,0)$ (respectively, for $\sgn c = (-,+,0,-)$), we have $L= \frac{c_2}{c_3} = \frac{\Delta(341)}{\Delta(214)} = \frac{a_1-a_2}{a_3 - a_1}$ (respectively, $L= \frac{c_1}{c_4} = \frac{\Delta(243)}{\Delta(123)} = \frac{a_1-a_2}{a_2 - a_4}$), the condition \eqref{eq:c3_or_c4_zero_permanence_condition} holds whenever $P_1$ (respectively, $P_2$) is close enough to $P_3$ (respectively, to $P_4$). Further, the condition \eqref{eq:c3_or_c4_zero_permanence_condition} holds whenever the slope of $P_1P_2$ is at most $-1$.

The next lemma covers the cases, where $\sgn c$ is one of $(+,0,-,0)$ and $(0,+,0,-)$, i.e., the situations in the right panels in the $4$th and $5$th rows in Figure~\ref{fig:9_cs}. It is the case (C1c) in Theorem \ref{thm:oppositesign}. Once this lemma is proven, it also concludes the proof of Theorem \ref{thm:oppositesign}.

\begin{lemma} \label{lem:P1=P3}
Assume $\sgn J = \begin{pmatrix}
+ & - \\ + & -
\end{pmatrix}$, $\det J > 0$, and that $a_4 \leq a_2 < a_1 \leq a_3$ and $\sgn c$ is either $(+,0,-,0)$ or $(0,+,0,-)$. Then the ODE~\eqref{eq:ode_exp} is permanent if and only if $a_1-a_2 + b_1 - b_2 \leq 0$ and $\tr J < 0$.
\end{lemma}
\begin{proof}
We first prove the case $\sgn c = (+,0,-,0)$. Note that then $c_1 = -c_3$, $a_1=a_3$, and $b_1=b_3$. The surface $\overline S$ is thus the triangle $\{x \in \Delta_4~|~ x_1=x_3\}$. The dynamics on $\overline S$ is given by the replicator dynamics for the strategies $\frac12 E_1+\frac12 E_3, E_2, E_4$ with matrix
\begin{align}\label{eq:m24_matrix}
\begin{pmatrix}
0                           & a_2-a_1 & b_4- b_1 \\
\frac{a_2-a_1 + b_2-b_1}{2} & 0       & b_4- b_2 \\
\frac{a_4-a_3 + b_4-b_3}{2} & a_2-a_4 & 0        \\
\end{pmatrix}.
\end{align}
Let $E_m = \frac12 E_1+\frac12 E_3$. See Figure~\ref{fig:P1=P3}.

We now show that $a_1 - a_2 + b_1 - b_2 > 0$ implies that $E_m$ is asymptotically stable, contradicting permanence. Indeed, a short calculation shows that $a_1 - a_2 + b_1 - b_2 > 0$, $b_1-b_2<0$, $b_3-b_4<0$ and $\det J > 0$ imply that $a_3-a_4 + b_3-b_4 > 0$, and hence, both of the eigenvalues at $E_m$ in $\overline S$ are negative.

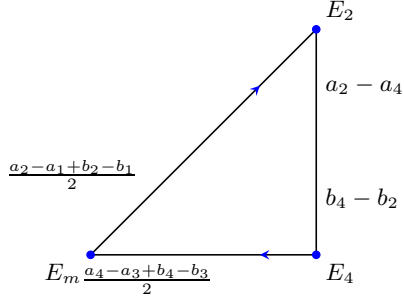
\begin{figure}
\begin{center}
\begin{tikzpicture}[scale=1,radius=0.05]
\coordinate (Em)  at (0,0);
\coordinate (E2)  at (3,3);
\coordinate (E4)  at (3,0);
\coordinate (arr2m) at (2.25,2.25);
\coordinate (arr4m) at (2.25,0);
\coordinate (eig24) at (3,2.25);
\coordinate (eig42) at (3,0.75);
\coordinate (eigm2) at (0.75,0.75);
\coordinate (eigm4) at (0.75,0);
\draw [semithick] (E1) -- (E4);
\draw [semithick] (E4) -- (E2);
\draw [semithick] (E2) -- (E1);
\node [below left]  at (Em) {\small $E_m$};
\node [above right] at (E2) {\small $E_2$};
\node [below right] at (E4) {\small $E_4$};
\node [above left]  at (eigm2) {\small $\frac{a_2-a_1+b_2-b_1}{2}$};
\node [below]       at (eigm4) {\small $\frac{a_4-a_3+b_4-b_3}{2}$};
\node [right]       at (eig42) {\small $b_4-b_2$};
\node [right]       at (eig24) {\small $a_2-a_4$};
\draw [fill,blue] (Em) circle;
\draw [fill,blue] (E2) circle;
\draw [fill,blue] (E4) circle;
\draw [blue,draw opacity=0,semithick,-stealth] (Em) -- (arr2m);
\draw [blue,draw opacity=0,semithick,-stealth] (E4) -- (arr4m);
\end{tikzpicture}
\end{center}
\caption{The flow in the triangle $x_1=x_3$ in case $P_1 = P_3$.}
\label{fig:P1=P3}
\end{figure}

From now on, we assume that $a_1 - a_2 + b_1 - b_2 \leq 0$. Thus, there is no equilibrium on the edge $\mathcal{F}_{2m}$ and the flow goes from $E_m$ to $E_2$. We claim that
\begin{align}\label{eq:claim_Em_does_not_attract}
E_m \text{ does not attract an orbit from }S.
\end{align}
Suppose on the contrary that $E_m$ attracts an orbit from $S$. Then both eigenvalues must be less than or equal to zero. Thus, $a_2-a_1 + b_2-b_1 = 0$ and $a_4-a_3 + b_4-b_3 \leq 0$. As a short calculation shows, $\det J \neq 0$ implies that at least one eigenvalue at $E_m$ is nonzero. Hence, $a_4-a_3 + b_4-b_3 < 0$. The center manifold at $E_m$ is thus $1$-dimensional (and is contained in $\mathcal{F}_{m2}$). By the reduction principle (see~\cite[Theorem~5.2]{kuznetsov:2004}), $E_m$ is topologically a saddle and cannot attract any orbit from $S$. This contradiction proves the claim \eqref{eq:claim_Em_does_not_attract}.

If there exists an equilibrium $E_{24}$ on the edge $\mathcal{F}_{24}$ (i.e., $a_4 < a_2$ and $b_2<b_4$), it cannot attract any orbit from the interior, see Lemma \ref{lem:edge_eq_doesnt_attract}. (Note that $P_2=P_4$ would contradict $c_1\neq0$.) Note also that $b_2 < b_4$ implies $\tr J < 0$. Indeed, 
\begin{align*}
\tr J = a_1 - a_2 + b_3 - b_4 < a_1 - a_2 + b_1 - b_2 \leq 0.
\end{align*}

If there exists an equilibrium $E_{m4}$ on the edge $\mathcal{F}_{m4}$ (i.e., $a_4-a_3 + b_4-b_3 > 0$), the eigenvalue $\Gamma_{m4}^2$ in $\overline S$, by \eqref{eq:Gammaijk} applied to \eqref{eq:m24_matrix}, has the same sign as $-c_3$, and thus, is positive. Therefore, $E_{m4}$ (if it exists at all) cannot attract any orbit from the interior. Note also that $a_4-a_3 + b_4-b_3 > 0$ implies $\tr J < 0$. Indeed,
\begin{align*}
\tr J = a_1 - a_2 + b_3 - b_4 < a_1 - a_2 + a_4 - a_3 = a_4 - a_2 < 0. 
\end{align*}

None of $E_m$, $E_2$, $E_4$ can attract an orbit from $S$ (by the claim \eqref{eq:claim_Em_does_not_attract} and Lemma \ref{lem:corners_doesnt_attract}). By Theorem \ref{thm:butler}, permanence follows if at least one of $E_{24}$ and $E_{m4}$ exists.

It remains to characterize permanence when none of $E_{24}$ and $E_{m4}$ exists. This leads to a heteroclinic cycle along $E_m$, $E_2$, $E_4$, $E_m$, as in the rock-paper-scissors game. Let us define $L_\infty$ by
\begin{align*}
L_\infty = \frac{a_2-a_1 + b_2-b_1}{2}(a_2-a_4)(b_4- b_1) + \frac{a_4-a_3 + b_4-b_3}{2}(a_2-a_1)(b_4- b_2).
\end{align*}
(Note that $L_\infty$ is the determinant of \eqref{eq:m24_matrix}.) With this, if $L_\infty > 0$ then $\partial S$ is repelling (i.e., the ODE~\eqref{eq:ode_exp} is permanent), and if $L_\infty < 0$ then $\partial S$ is attracting  (i.e., the ODE~\eqref{eq:ode_exp} is not permanent), see \cite[Theorem 3]{hofbauer:1981}. A short calculation shows that $L_\infty = -\frac{c_1}{2} \tr J$. If $\tr J = 0$ then the system is not permanent, because we have a global center (see \cite[Theorem 6]{zeeman:1980}, \cite[Theorem 7]{boros:hofbauer:mueller:regensburger:2017b}, \cite[Theorem 7.7.2]{HS98}, \cite[Exercise 16.5.5(e)]{HS88}).

To prove the statement in the case $\sgn c = (0,+,0,-)$ one can argue similarly. However, it is more elegant to apply the rotation $\mathbf{r}_2$ to the case $\sgn c = (+,0,-,0)$ (the rotation $\mathbf{r}_2$ leaves both $a_1-a_2+b_1-b_2$ and $\tr J$ invariant).
\end{proof}

We remark that whenever the ODE~\eqref{eq:ode_exp} with the assumptions of Lemma \ref{lem:P1=P3} is permanent, the origin is globally asymptotically stable. This follows from the classification of the replicator dynamics on the triangle $\Delta_3$, see \cite{zeeman:1980,HS88,HS98}.

\section{Examples}

We illustrate some of our results via two examples. The first one is actually a special case of the second one.

\subsection{Selkov's glycolytic oscillation}

Selkov \cite{Selkov68} considered the planar S-system
\begin{align}\label{eq:ode_selkov}
\begin{split}
\dot x &= 1 - xy^{\gamma},  \\
\dot y &= k ( xy^{\gamma} - y)
\end{split}
\end{align}
with $k>0$ and $\gamma \in \mathbb{R}$ as a model for glycolytic oscillations. One can rewrite it as the ODE~\eqref{eq:ode_exp} with
\begin{align*}
\begin{aligned}
a_1 &= -1, &b_1 &= 0,           \\
a_2 &=  0, &b_2 &= k\gamma,     \\
a_3 &=  1, &b_3 &= k(\gamma-1), \\
a_4 &=  0, &b_4 &= 0.
\end{aligned}
\end{align*}
Then the Jacobian of the ODE~\eqref{eq:ode_exp} at the origin is given by
\begin{align*}
J = \begin{pmatrix}
-1 & -k\gamma \\
1& k(\gamma-1)
\end{pmatrix},
\end{align*}
while $c = k(\gamma,1-\gamma,\gamma,-1-\gamma)$. Thus, for $\gamma < 1$ (respectively, for $\gamma = 1$) permanence follows from case (A) (respectively, from case (B2)) in Theorem \ref{thm:samesign}. For $\gamma > 1$, the system is not permanent, because based on $\sgn J$ it falls under case (C3) in Theorem \ref{thm:oppositesign}, but $\sgn (c_1,c_2) = (+,-) \notin \{(+,+), (0,+), (+,0)\}$. In the ODE~\eqref{eq:A}, the corner $E_4$ is asymptotically stable and in the ODE~\eqref{eq:ode_selkov}, some orbits go to infinity along the $x$-axis. For $\gamma \leq 1+\frac{1}{k}$, the origin is asymptotically stable, it undergoes a supercritical Andronov-Hopf bifurcation at $\gamma = 1+\frac{1}{k}$, see \cite{Selkov68,boros:hofbauer:mueller:2017,brechmann:rendall:2018}. It is shown that the ODE~\eqref{eq:ode_selkov} can have at most one limit cycle and it is an open question, whether this limit cycle disappears in a heteroclinic bifurcation at some value $\widehat{\gamma}(k)$ and for $\gamma > \widehat{\gamma}(k)$ all orbits (except the unique positive equilibrium) escape to infinity, see \cite{brechmann:rendall:2018}.

\subsection{The Lotka reactions with generalized mass-action kinetics}

Dancs\'o et al.~\cite{dancso:farkas:farkas:szabo:1991} studied the Lotka reactions with generalized mass-action kinetics. Here we consider the special case
\begin{align}\label{eq:ode_lotka}
\begin{split}
\dot x &= x^{\alpha}  - x y^{\beta}, \\
\dot y &= k(x y^{\beta} - y)
\end{split}
\end{align}
with $k > 0$ and $\alpha$, $\beta \in \mathbb{R}$. In \cite{boros:hofbauer:mueller:2017}, we showed that the unique positive equilibrium of the ODE~\eqref{eq:ode_lotka} is globally asymptotically stable for all $k>0$ if and only if
\begin{align*}
\alpha \leq 1, \beta \leq 1, (\alpha,\beta) \neq (1,1), \text{ and } \alpha\beta > \alpha - 1,
\end{align*}
while it is globally asymptotically stable for $k = 1$ if and only if either
\begin{align*}
&\alpha \leq 1, \beta \leq 1, (\alpha,\beta) \neq (1,1), \text{ and } \alpha\beta > \alpha - 1 \text{ or}\\
&1 < \alpha \leq \frac{3}{2} \text{ and } \alpha-1 \leq \beta \leq 2-\alpha.
\end{align*}

Now we characterize permanence for fixed $k>0$ and $\alpha$, $\beta \in \mathbb{R}$ (in particular for $k = 1$ and $\alpha$, $\beta \in \mathbb{R}$). Further, based on this, we characterize those exponents $\alpha$, $\beta \in \mathbb{R}$ for which permanence holds for all $k>0$.

One can rewrite the ODE~\eqref{eq:ode_lotka} as the ODE~\eqref{eq:ode_exp} with
\begin{align} \label{eq:ab_in_alphabeta}
\begin{aligned}
a_1 &= \alpha-1, &b_1 &= 0,          \\
a_2 &=        0, &b_2 &= k\beta,     \\
a_3 &=        1, &b_3 &= k(\beta-1), \\
a_4 &=        0, &b_4 &= 0.
\end{aligned}
\end{align}
Then the Jacobian of the ODE~\eqref{eq:ode_exp} at the origin is given by
\begin{align*}
J = \begin{pmatrix}
\alpha-1 & -k\beta \\
1& k(\beta-1)
\end{pmatrix},
\end{align*}
while $c = k(\beta,(\alpha-1)(\beta-1),-(\alpha-1)\beta,\alpha-1-\beta)$.

\begin{proposition}
The following three statements hold.
\begin{enumerate}[(i)]
\item\label{it:lotka1} The ODE~\eqref{eq:ode_exp} with \eqref{eq:ab_in_alphabeta} is permanent if and only if either
\begin{align*}
&\alpha \leq 1, \beta \leq 1, (\alpha,\beta) \neq (1,1), \text{ and } \alpha\beta > \alpha - 1,\\
&1 < \alpha < 2 \text{ and } \alpha-1 < \beta < 1, \text{ or} \\
&1 < \alpha < 2, \beta = \alpha-1, \text{ and } k > \beta(1-\beta)^{\frac{1-\beta}{\beta}}.
\end{align*}
\item\label{it:lotka2} The ODE~\eqref{eq:ode_exp} with \eqref{eq:ab_in_alphabeta} is permanent for $k=1$ if and only if either
\begin{align*}
&\alpha \leq 1, \beta \leq 1, (\alpha,\beta) \neq (1,1), \text{ and } \alpha\beta > \alpha - 1 \text{ or}\\
&1 < \alpha < 2 \text{ and } \alpha-1 \leq \beta < 1.
\end{align*}
\item\label{it:lotka3} The ODE~\eqref{eq:ode_exp} with \eqref{eq:ab_in_alphabeta} is permanent for all $k>0$ if and only if either
\begin{align*}
&\alpha \leq 1, \beta \leq 1, (\alpha,\beta) \neq (1,1), \text{ and } \alpha\beta > \alpha - 1 \text{ or}\\
&1 < \alpha < 2 \text{ and } \alpha-1 < \beta < 1.
\end{align*}
\end{enumerate}
\end{proposition}
\begin{proof}
Clearly, \eqref{it:lotka2} and \eqref{it:lotka3} follow from \eqref{it:lotka1}. Thus, it remains to prove \eqref{it:lotka1}.

If $\alpha \geq 1$ and $\beta \geq 1$ then there is no negative entry on the diagonal of $J$, so the system is not permanent, see Theorem \ref{thm:samesign}.

If $\alpha \leq 1$, $\beta \leq 1$, and $(\alpha,\beta)\neq(1,1)$ then there is no positive entry on the diagonal of $J$ (and at least one diagonal entry is negative), and the system falls under one of the cases (A), (B1), (B2) in Theorem \ref{thm:samesign} \eqref{it:samesign_charact}. Hence, the system is permanent.

For all other pairs $(\alpha,\beta)$, the assumptions of Theorem \ref{thm:oppositesign} hold. If $\alpha < 1$ and $\beta > 1$ then the system falls under case (C3) and is not permanent, because $c_2<0$.

If $\alpha > 1$ and $\beta < 1$ then the system falls under case (C1) in Theorem \ref{thm:oppositesign}. To have $\det J > 0$, we need $\beta > 0$. Therefore, $\sgn(c_1,c_2,c_3)= (+,-,-)$. Further, $c_4 \leq 0$ if and only if $\beta \geq \alpha-1$. When $\beta > \alpha-1$, the system is permanent as it falls under case (C1a). When $\beta = \alpha-1$, the system is permanent if and only if $k > \beta(1-\beta)^{\frac{1-\beta}{\beta}}$, see case (C1b).
\end{proof}

We now describe the behaviour in the strip $1 < \alpha < 2 \text{ and } \alpha-1 < \beta$ under $k = 1$. At the line $\beta = 2-\alpha$, a supercritical Andronov-Hopf bifurcation occurs, a stable limit cycle appears for $\beta$ slightly larger than $2-\alpha$, see \cite{boros:hofbauer:mueller:2017}. Because of permanence, an asymptotically stable limit cycle must exist for all $2-\alpha<\beta < 1$. At $\beta = 1$, $\sgn J = \begin{pmatrix} + & - \\ + & 0 \end{pmatrix}$, the divergence is positive, therefore, no closed orbit exists. Since $\sgn c = (+,0,-,-)$, $\partial S$ consists of the two edges $\mathcal{F}_{32}$, $\mathcal{F}_{24}$ and the curve $\mathcal{C}^1_{43}$, it is a heteroclinic cycle. The equilibrium $(1,1)$ is a global repeller and the heteroclinic cycle is the global attractor. Therefore, the stable limit cycle for $\beta < 1$ merges with the heteroclinic cycle at $\beta = 1$. For $\beta > 1$, Theorem \ref{thm:heteroclinic} applies, $L_\infty<0$ (note that one of the outgoing eigenvalues is zero: $b_1 = b_4$, so at $E_4$ in the direction $E_1$), and therefore $\partial S$ is strongly attracting. The equilibrium $(1,1)$ is a global repeller, since the divergence is positive.

\section{Three limit cycles}

In the papers \cite{dancso:farkas:farkas:szabo:1991} and \cite{boros:hofbauer:mueller:regensburger:2017,boros:hofbauer:mueller:regensburger:2017b}, examples of planar S-systems with one and two limit cycles were constructed around the equilibrium, respectively. Based on the findings of the present paper, we can construct one more limit cycle, one that is created near infinity.

Consider the ODE~\eqref{eq:ode_exp} with
\begin{align} \label{eq:ab_limit_cycle}
\begin{aligned}
a_1 &=   0, &b_1 &= 0,          \\
a_2 &=  -8, &b_2 &= 35,     \\
a_3 &=  10, &b_3 &= 20, \\
a_4 &= -20, &b_4 &= 28.
\end{aligned}
\end{align}
Then $\tr J = 0$ and $\det J > 0$. As in \cite[Section 4.3]{boros:hofbauer:mueller:regensburger:2017b}, the first focal value, $L_1$, is given by
\begin{align*}
L_1 = - \frac{\pi}{8} \frac{(b_3 - b_4) \left[ D b_2 - (a_3 - a_4) b_3 b_4 \right]}{b_2\sqrt{\det J}},
\end{align*}
where
\begin{align*}
D = a_3 a_4 + a_3 b_4 - a_4 b_3.
\end{align*}
Then, with the substitutions \eqref{eq:ab_limit_cycle}, we have $D = 480$, $L_1 = 0$, and the second focal value, $L_2$, is positive, according to the formula derived in \cite[Section 4.3]{boros:hofbauer:mueller:regensburger:2017b}. Further, from the definition in case (A) in Theorem \ref{thm:heteroclinic}, $L_\infty = 0$ (i.e., we do not know the behaviour near infinity).

Next, we perturb $b_2$ to a slightly smaller value $35 - \eps$ (with $\eps > 0$ small).
As a result, 
\begin{align}
\label{p-L1}
\begin{split}
\tr J &= 0,\\
L_1 &= -\frac{\pi}{(35-\eps)\sqrt{986}} D\eps < 0,\\
L_2 &> 0
\end{split}
\end{align}
and
\begin{align}\label{p-Linfty}
L_{\infty} = -a_2 D \eps > 0.
\end{align}
From \eqref{p-L1} we get a Bautin bifurcation (see~\cite[Section~8.3]{kuznetsov:2004}) near the origin and an unstable limit cycle $\Gamma_1$ is created. The origin is now asymptotically stable.
From \eqref{p-Linfty}, the system is permanent, see Theorem \ref{thm:heteroclinic}. Thus for large initial points, the solutions spiral inwards. By the Poincar\'e--Bendixson theorem, its $\omega$--limit set is nonempty and either contains an equilibrium (which is not possible, since the only equilibrium is the origin and it is surrounded by $\Gamma_1$), or is a periodic orbit, call it $\Gamma_{\infty}$. This $\Gamma_{\infty}$ must surround an equilibrium, i.e., the origin, and  must be attracting at least from the outside, so it is different from $\Gamma_1$.
  
Finally, after fixing $\eps >0$ with the above behaviour, we perturb for example $a_2$ to $a_2 - \mu$. Then $\mu = \tr J$ and an Andronov--Hopf bifurcation occurs at $\mu = 0$. It is supercritical, since $L_1 < 0$. So for $\mu > 0$, the origin is unstable, and a small stable limit cycle $\Gamma_0$ is created.

Thus this system has (at least) two stable limit cycles (if $\Gamma_{\infty}$ is repelling towards the interior then there will be some other limit cycle attracting from both sides), and at least one unstable one. 

\appendix
\section{Replicator dynamics}\label{sec:app_replicator}

In this section, we collect some general facts about the replicator dynamics, i.e., about the ODE
\begin{align} \label{eq:ode_A_n}
\dot x_i = x_i \left[ (Ax)_i - x^\mathsf{T}A x\right] \text{ for } i = 1, \ldots, n
\end{align}
with (the $n-1$-dimensional) state space $\Delta_n = \{x \in \mathbb{R}^n_{\geq0}~|~x_1 + \cdots + x_n = 1\}$. We assume (w.l.o.g.) throughout that $a_{ii} = 0$ for all $i = 1, \ldots, n$.

For $k \in \{1, \ldots, n\}$, we denote by $E_k$ the $k$th corner of $\Delta_n$, i.e., the vector whose $k$th coordinate is $1$ and all the others are $0$. All the corners are equilibria. For $l \neq k$, the eigenvalue at $E_k$ in the direction $E_l$ is $a_{lk}$.

For $i,j \in \{1,\ldots,n\}$ with $i \neq j$, we denote by $\mathcal{F}_{ij}$ the edge $\{x \in \Delta_n ~|~ x_i + x_j = 1\}$. There exists a unique edge equilibrium $E_{ij}$ in $\mathcal{F}_{ij}$ (strictly between $E_i$ and $E_j$) if and only if $\sgn a_{ij} = \sgn a_{ji} \neq 0$. If $E_{ij}$ exists, its $i$th and $j$th coordinates are given by $\frac{a_{ij}}{a_{ij}+a_{ji}}$ and $\frac{a_{ji}}{a_{ij}+a_{ji}}$, respectively. The internal eigenvalue (within the edge $\mathcal{F}_{ij}$) at $E_{ij}$ is given by $-\frac{a_{ij}a_{ji}}{a_{ij}+a_{ji}}$. For $k \neq i,j$, the external eigenvalue at $E_{ij}$ in the direction $E_k$ is given by
\begin{align}\label{eq:Gammaijk}
\Gamma_{ij}^k = \frac{\dot x_k}{x_k}\Big|_{x=E_{ij}} = \frac{a_{ki}a_{ij}+ a_{kj}a_{ji}-a_{ij}a_{ji}}{a_{ij}+a_{ji}},
\end{align}
see \cite[(20.17)]{HS88}.

Let $\widehat x \in \Delta_n$ be an equilibrium of the ODE~\eqref{eq:ode_A_n} with support $I \subseteq \{1, \ldots, n\}$. Then $(A\widehat x)_i = \widehat x^\mathsf{T} A \widehat x$ for $i \in I$ and $\widehat x_i = 0 $ for $i \notin I$. The equilibrium $\widehat x$ is said to be \emph{saturated} if $(A\widehat x)_i \leq \widehat x^\mathsf{T} A \widehat x$ for all $i \notin I$. Note that $(A\widehat x)_i - \widehat x^\mathsf{T} A \widehat x$ is the external eigenvalue at $\widehat{x}$ in the direction $i$. Therefore, an equilibrium is saturated if and only if all the external eigenvalues are non-positive. In particular, we have the following lemma.

\begin{lemma}
Consider the ODE~\eqref{eq:ode_A_n} with $a_{ii} = 0$ for all $i = 1,\ldots,n$. Then the following statements hold.
\begin{enumerate}[(i)]
\item Every interior equilibrium is saturated.
\item The corner $E_k$ is saturated if and only if $a_{ik} \leq 0$ for all $i\neq k$.
\item If there exists a unique edge equilibrium $E_{ij}$ in the edge $\mathcal{F}_{ij}$, it is saturated if and only if $\Gamma_{ij}^k\leq0$ for all $k \neq i,j$.
\end{enumerate}
\end{lemma}

For a proof of the following lemma, see \cite[Theorem 7.2.1]{HS98}.
\begin{lemma} \label{lem:saturated}
If an interior orbit of the ODE~\eqref{eq:ode_A_n} converges to an equilibrium $\widehat x$ on the boundary of $\Delta_n$, as $t \to \infty$, then $\widehat x$ is saturated.
\end{lemma}

\section{The replicator dynamics with matrix \eqref{eq:A}} \label{sec:app_special}

Assuming there exists a unique edge equilibrium $E_{ij}$ on the edge $\mathcal{F}_{ij}$ for the ODE~\eqref{eq:ode_A} with matrix \eqref{eq:A}, we compute the corresponding external eigenvalues. The formula \eqref{eq:Gammaijk} gives
\begin{align*}
\Gamma_{12}^3 &= \Gamma_{12}^4 = \Gamma_{34}^1 = \Gamma_{34}^2 = 0, \\
\Gamma_{13}^k &= \frac{1}{(b_1-b_3)+(a_3-a_1)}\cdot\begin{cases} (-c_4), & k = 2, \\ (+c_2), & k = 4, \end{cases} \\
\Gamma_{23}^k &= \frac{1}{(b_2-b_3)+(a_2-a_3)}\cdot\begin{cases} (-c_4), & k = 1, \\ (+c_1), & k = 4, \end{cases} \\
\Gamma_{24}^k &= \frac{1}{(b_4-b_2)+(a_2-a_4)}\cdot\begin{cases} (-c_3), & k = 1, \\ (+c_1), & k = 3, \end{cases} \\
\Gamma_{14}^k &= \frac{1}{(b_4-b_1)+(a_4-a_1)}\cdot\begin{cases} (-c_3), & k = 2, \\ (+c_2), & k = 3, \end{cases} \\
\end{align*}
where $c_1$, $c_2$, $c_3$, $c_4$ are given by the equations \eqref{eq:c}. Assuming $a_4 \leq a_2 < a_1 \leq a_3$, we have
\begin{align}\label{eq:Gamma_special}
\begin{split}
\begin{aligned}
\sgn\Gamma_{13}^2 &= -\sgn c_4, \\
\sgn\Gamma_{23}^1 &= +\sgn c_4, \\
\sgn\Gamma_{24}^1 &= -\sgn c_3, \\
\sgn\Gamma_{14}^2 &= +\sgn c_3, \\
\end{aligned}
\hspace{1cm}
\begin{aligned}
\sgn\Gamma_{13}^4 &= +\sgn c_2, \\
\sgn\Gamma_{23}^4 &= -\sgn c_1, \\
\sgn\Gamma_{24}^3 &= +\sgn c_1, \\
\sgn\Gamma_{14}^3 &= -\sgn c_2. \\
\end{aligned}
\end{split}
\end{align}




\bibliographystyle{abbrv}

\bibliography{stability,perm}


\end{document}